\documentclass[titlepage,a4paper,11pt]{amsart} 
\usepackage[foot]{amsaddr}
\usepackage{amssymb}
\usepackage[lite]{amsrefs}
\usepackage{t1enc}
\usepackage[latin1]{inputenc}
\usepackage[english]{babel}
\usepackage{mathrsfs} 
\usepackage{hyperref}
\usepackage[all]{xy} 
\usepackage[left=3cm,top=3cm,right=3cm,bottom=3cm]{geometry}

\numberwithin{equation}{section}

\newtheorem{thm}{Theorem}[section]
\newtheorem{prop}[thm]{Proposition}
\newtheorem{lem}[thm]{Lemma}
\newtheorem{cor}[thm]{Corollary}
\newtheorem{defn}[thm]{Definition}
\newtheorem{rem}[thm]{Remark}
\newtheorem{ex}[thm]{Example}

\DeclareMathOperator\Aut{Aut}
\DeclareMathOperator\Ad{Ad}

\DeclareMathOperator\dd{d}

\DeclareMathOperator\id{id}

\renewcommand{\tilde}{\widetilde}

\renewcommand{\bar}{\overline}
\newcommand{\oline}{\overline}
\newcommand{\eset}{\emptyset} 
\renewcommand{\:}{\colon} 
\newcommand{\eps}{\varepsilon} 
\newcommand{\subeq}{\subseteq} 
\newcommand{\bone}{\mathbf 1}
\newcommand{\pmat}[1]{\begin{pmatrix} #1 \end{pmatrix}}

\def\g{{\mathfrak g}}

\def\q{{\mathfrak q}}
\def\fq{{\mathfrak q}}

\def\h{{\mathfrak h}}

\def\p{{\mathfrak p}}
\def\gl{{\mathfrak{gl}}}

\DeclareMathOperator\Ev{ev}

\DeclareMathOperator\diag{diag}
\DeclareMathOperator\Hom{Hom}
\DeclareMathOperator\tr{tr}
\DeclareMathOperator\deter{det}

\DeclareMathOperator\spanh{span}

\def\Z{\mathbb{Z}}
\def\V{\mathbb{V}}
\def\bV{\mathbb{V}}

\def\L{\mathcal L}
\def\D{\mathcal D}
\def\K{\mathcal K}
\def\A{\mathcal A}
\def\cA{\mathcal A}

\def\M{\mathcal M}
\def\cM{\mathcal M}
\def\H{\mathcal H}
\def\cH{\mathcal H}

\def\cK{\mathcal K}
\def\cE{\mathcal E}
\def\O{\mathcal O}
\def\cO{\mathcal O}
\def\B{\mathcal B}
\def\Fo{\mathcal F}

\def\cR{\mathcal R}

\DeclareMathOperator\Gr{Gr}
\DeclareMathOperator\GL{GL}
\DeclareMathOperator\U{U}
\DeclareMathOperator\Ort{O}

\DeclareMathOperator\res{res}

\def\C{\mathbb{C}}
\def\R{\mathbb{R}}
\def\N{\mathbb{N}}
\def\T{\mathbb{T}}
\def\P{\mathbb{P}}

\title[Multiplicity freeness in sections of 
holomorphic Hilbert bundles]{Multiplicity freeness of unitary representations in sections of 
holomorphic Hilbert bundles }
\date{\today}
\author{Mart\'in Miglioli}
\email[Mart\'in Miglioli]{martin.miglioli@gmail.com}
\address[Mart\'in Miglioli]{IAM-CONICET. Saavedra 15, Piso 3, (1083) Buenos Aires, Argentina}
\thanks{The first author was supported by IAM-CONICET, grants PIP 2010-0757 (CONICET), 2010-2478 (ANPCyT) and a DAAD grant for short term visit}

\author{Karl-Hermann Neeb} 
\email[Karl-Hermann Neeb]{neeb@math.fau.de}
\address[Karl-Hermann Neeb]{Department Mathematik, 
University Erlangen-Nuremberg, Cauerstrasse 11, 
91058 Erlangen, Germany}
\thanks{The second author was supported in part by DFG-grant NE 413/9-1.}

\begin{document}

\begin{abstract}
We  prove several results asserting that the action of a 
Banach--Lie group on Hilbert spaces of holomorphic sections of a 
holomorphic Hilbert space bundle over a complex Banach manifold is multiplicity free. 
These results require the existence of compatible antiholomorphic bundle maps 
and certain multiplicity freeness assumptions for stabilizer groups. 
For the group action on the base, the notion of an $(S,\sigma)$-weakly visible 
action (generalizing T.~Koboyashi's visible actions) 
provides an effective way to express the assumptions in an economical fashion. 
In particular, we derive a version for group actions on homogeneous bundles 
for larger groups. We illustrate these general results by 
several examples related to operator groups and von Neumann algebras. 

\medskip

\end{abstract}

\maketitle
\tableofcontents

\section{Introduction}

A unitary representation $(\pi,\cH)$ of a group $G$ 
on a complex Hilbert space $\cH$ is called 
 \textit{multiplicity-free} if its commutant, 
the von Neumann algebra $\pi(G)'$ of continuous $G$-inter\-twining operators,  
is commutative. 
Multiplicity-free representations are special in the sense that one may 
expect to find natural decompositions into irreducible ones based 
on direct integrals over the spectrum of the commutant. 
We refer to \cite{kob2} and the reference therein for a survey of multiplicity 
free theorems and its applications in the context of finite dimensional Lie groups. 

The main results of this paper consist in 
propagation theorems for the multiplicity-free property (MFP) 
from the representation of a stabilizer group in a fiber Hilbert space 
to Hilbert spaces of holomorphic sections of  holomorphic Hilbert bundles. 
Our results extend those of T.~Kobayashi concerning 
finite-dimensional bundles \cite{kob}  
to Hilbert bundles over Banach manifolds. More specifically, 
the fibers are complex Hilbert spaces
 and the groups act as fiberwise isometric holomorphic 
bundle automorphisms. We apply these propagation theorems to branching problems 
of representations of infinite dimensional groups constructed by 
holomorphic induction. 
Here an essential part is that the isotropy representations 
are not finite dimensional, hence in general not direct sums of irreducible 
representations. As we shall see, 
this difficulty can be overcome by working systematically 
with the commutant as a von Neumann algebra. 

A variant of the propagation theorem is formalized as in 
\cite{kob} in terms of so-called visible actions. 
The $G$-action on $M$ is called $(S,\sigma^M)$-weakly visible 
if $S\subseteq M$ is a subset for which the closure of $G.S$ has interior points 
and $\sigma^M$ is an antiholomorphic diffeomorphism 
of $M$ preserving all $G$-orbits through~$S$ and leaving $S$ invariant.
If $\sigma^M$ lifts to an anti-holomorphic bundle endomorphism 
$\sigma^{\V}$ which is compatible with the $G$-action with respect to an 
automorphism $\sigma^G$ of $G$ satisfying $\sigma^\bV(g.v) = \sigma^G(g).\sigma^\bV$, 
then one can formulate a variant of the propagation 
theorem (Theorem~\ref{multfreevis}) asserting the multiplicity-freeness 
of the $G$-representation on $\cH \subseteq \Gamma(\bV)$ if, for every 
$s \in S$, the antiunitary operator $\sigma^\bV_s$ commutes on 
$\bV_s$ with the hermitian part of the commutant of the $G_s$-action. 
A third form of the propagation theorem is obtained in the setting where 
the bundle $\V=G\times_{\rho,H} V$ 
is associated to a homogeneous $H$-principal bundle $G\to M=G/H$ by a norm 
continuous unitary representation $(\rho,V)$ of~$H$. 
We plan to use this formulation for 
concrete branching problems in the representation theory of Banach--Lie groups. 

In the infinite dimensional context there are no general results on the existence of 
solutions of $\bar{\partial}$-equations that can be used to verify integrability 
of complex structures on Banach manifolds and in particular on vector bundles. 
Here  \cite{neeb13a} 
provides effective methods to treat unitary representations of Banach--Lie groups in spaces of holomorphic sections of homogeneous Hilbert bundles. 
For a real homogeneous Banach vector bundle $\V=G\times_H V$ 
over $G/H$ associated to a norm continuous representation of~$H$, 
compatible complex structures 
are obtained by extensions $\beta:\q\to \gl (V)$ 
of the differential $\dd\rho\colon\h\to \gl(V)$ to a representation of the complex 
subalgebra $\q\subseteq \g_{\C}$  specifying the complex structure 
on $M$ by $T_{1H}M\simeq\g_{\C}/\q$.
 
Another particularity of the infinite dimensional context is that in general self-adjoint operators are not conjugate to operators in a given maximal abelian subalgebra. This is 
well known for the algebra of bounded operators acting on a Hilbert space, see \cite{bw} for the case of Hilbert-Schmidt operators and the example after 
\cite{ak06}*{Thm.~5.4} for operators in a finite von Neumann factor. In the fundamental examples the group action is derived from the two-sided action of 
the unitary group $\U(\cM)$ on the algebra $\cM$ endowed with 
a scalar product derived from its trace, and the slice 
$S$ is derived from the hermitian operators in a maximal abelian subalgebra.  That 
the closure of $G.S$ has interior points comes 
from an approximate Cartan decomposition which follows from diagonalizabity on a dense subset. 
Also, since the isotropy representations in the infinite dimensional context may not be discretely decomposable as in the finite dimensional case treated in \cite{kob} we do not impose this condition in the propagation theorems. 

The structure of this paper is as follows. 
Section \ref{prel} contains some preliminary results on equivariant holomorphic Hilbert bundles and representations in Hilbert spaces of holomorphic sections.
In Section \ref{compgeomultfree} we start with our 
First Propagation Theorem~\ref{multfree1}, which asserts the multiplicity freeness 
of a unitary representation of a group $G$ in a Hilbert space 
of holomorphic sections, provided there exists an antiholomorphic 
bundle map satisfying certain compatibility conditions formulated in 
terms of stabilizer representations in points $m$ belonging to a subset 
$D\subeq M$.

For the sake of easier application of this result in the infinite dimensional 
context, we slightly extend T.~Kobayashi's notion of a visible action (\cite{kob2}). 
We call the action of the group $G$ on a complex manifold $M$ by holomorphic maps 
$(S,\sigma^M)$-weakly visible if $S \subeq G$ is a subset for which 
the closure of $G.S$ has interior points and $\sigma^M \: M \to M$ 
is an antiholomorphic diffeomorphism fixing $S$ pointwise and preserving the 
$G$-orbits through~$S$. For a visible action one requires in addition 
that $G.S$ is an open 
subset of $M$, but this is not satisfied in many interesting infinite 
dimensional situations, where the weak visibility can be verified. 
This leads us to our Second Propagation Theorem~\ref{multfreevis},
where the assumptions are formulated in terms of a weakly visible 
action. 

We then turn to the special case where 
the bundle $\bV$ is a homogeneous bundle over a homogeneous space 
$G/H$ of a Banach--Lie group~$G$. 
In Section \ref{assocbundles} we discuss 
$G$-invariant complex structures on such bundles and antiholomorphic 
isomorphisms. This is used in Section~\ref{propagassocbund} 
to obtain a propagation theorem for the multiplicity freeness 
of the representation of a subgroup $K \subeq G$ on Hilbert spaces 
of holomorphic sections of~$\bV$ (Theorem~\ref{multfreegroup}). 

In Section \ref{exampleswv} we eventually discuss various concrete situations 
in the Banach context, where the results of this paper apply naturally. 
In particular we exhibit several kinds of 
weakly visible actions on infinite dimensional spaces 
and  state some corresponding propagation theorems. 
A thorough investigation of these particular representations and 
concrete branching results are the topic of ongoing research.

\section{Preliminaries}\label{prel}

\subsection{Equivariant holomorphic Hilbert bundles}

Let $q:\V=\coprod_{m\in M}\V_m\to M$ be a holomorphic vector bundle over a connected 
complex Banach manifold $M$ whose fibers are complex Hilbert spaces. 
We write $\Gamma (\V)$ for the space of holomorphic sections of $\V\to M$.
We further assume that $G$ is a group (at this point no topology on $G$ is assumed) 
which acts on $\V$ by isometric holomorphic bundle automorphisms 
$(\gamma_g)_{g \in G}$. We denote the action of $G$ on 
the base space simply by $m\mapsto g.m$ for $g\in G$. 
In particular, we obtain for each $m \in M$
 a unitary representation 
\[ \rho_m:G_m\to \U(\V_m) \] 
of the isotropy subgroup $G_m:=\{g\in G:g.m=m\}$ on the fiber $\V_m$. 
Finally, the action of $G$ on the bundle $\V\to M$ gives rise to a representation 
$\delta$ of $G$ on $\Gamma (\V)$ by 
\begin{align}
(\delta_gs)(m):=\gamma_{g}(s(g^{-1}.m)) \quad\mbox{for}\quad g\in G\quad\mbox{and}\quad s\in\Gamma (\V).
\label{actionbundle}
\end{align}

\subsection{Reproducing kernels for Hilbert bundles}\label{repkern}

Let $q:\V\to M$ be a holomorphic Hilbert bundle on the complex manifold $M$. A Hilbert subspace $\H\subseteq \Gamma (\V)$ is said to have  \textit{continuous point evaluations} if all the evaluation maps 
$$\Ev_m:\H\to\V_m, \, s\mapsto s(m)$$
are continuous and the function $m\mapsto \|\Ev_m\|_{B(\H,\V_m)}$ is locally bounded. Then 
$$Q(m,n):=(\Ev_m)(\Ev_n)^*\in B(\V_n,\V_m),$$
defines a holomorphic section of the operator bundle 
\[ B(\V):=\coprod_{(m,n)\in M\times M^{\rm op}}B(\V_n,\V_m)\to M\times M^{\rm op},\]
where $ M^{\rm op}$ is the complex manifold $M$ endowed with the opposite complex structure. This section is the \textit{reproducing kernel} of $\H$. 

\begin{defn}{\rm 
We call a Hilbert subspace $\H\subseteq\Gamma (\V)$ with continuous point evaluations \textit{$G$-invariant}, if $\H$ is invariant under the action defined by (\ref{actionbundle}) and the so obtained representation $\pi$ of $G$ on $\H$ is unitary. In this case we say that $(\pi,\H)$ is \textit{realized} in $\Gamma (\V)$.} 
\end{defn}

\begin{lem}\label{lemeqkern}
\begin{enumerate}
\item For $j =1,2$, let $Q_j$ be the reproducing kernels of the Hilbert spaces 
$\H_j\subseteq\Gamma(\V)$ with inner products $\langle\cdot, \cdot \rangle_{\H_j}$. 
If $Q_1=Q_2$, then the subspaces $\H_1$ and $\H_2$ coincide and the inner products 
$\langle\cdot, \cdot\rangle_{\H_1}$ and $\langle\cdot, \cdot\rangle_{\H_2}$ are the same.
\item If $M$ is connected and $Q_1(m,m)=Q_2(m,m)$ for all $m$ in a subset $D$ which is dense in an open subset of $M$, then $Q_1=Q_2$.
\end{enumerate}
\end{lem}

\begin{proof}
We can represent holomorphic sections of the bundle $q:\V\to M$ by holomorphic functions on the total space of the dual bundle $\V^*$ which are linear on each fiber via the $G$-equivariant embedding
\[ \Psi:\Gamma(\V)\to\O(\V^*),\quad \Psi(s)(\alpha_m)=\alpha_m(s(m))\quad\mbox{for}\; s\in \Gamma(\V),\; \alpha_m\in \V_m^*,\] 
where $\O(\V^*)$ is the space of holomorphic functions on $\V^*$, see 
\cite{neeb13a}*{Remark 2.2}. For a reproducing kernel Hilbert space $\H\subseteq \Gamma(\V)$ with reproducing kernel $Q$ we obtain a reproducing kernel Hilbert space of holomorphic functions $\Psi(\H)\subseteq \O(\V^*)$ with reproducing kernel 
$$K(\alpha_m,\beta_n)=\Ev_{\alpha_m}\Ev_{\beta_n}^*\in B(\C,\C)\simeq \C\quad\mbox{for}\quad\alpha_m\in\V_m^*,\beta_n\in\V_n^*.$$
Since
\[  \Ev_{\alpha_m} \circ \Psi = \alpha_m \circ \Ev_m
\quad\mbox{for}\quad \alpha_m\in \V_m^*,\]
we obtain for $f:=\Psi(s)\in \Psi(\H)$,  $s\in \H$, the relation 
$$\Ev_{\alpha_m}(f)=\alpha_m(\Ev_m(\Psi^{-1}(f))),$$
so that 
\begin{align*}
K(\alpha_m,\beta_n)&=\Ev_{\alpha_m}\Ev_{\beta_n}^*
=(\alpha_m\circ\Ev_m\circ\Psi^{-1})(\beta_n\circ\Ev_n\circ\Psi^{-1})^*\\
&=\alpha_m\Ev_m\Ev_n^*\beta_n^*=\alpha_mQ(m,n)\beta_n^*.
\end{align*}
Therefore, if $\alpha_m=\langle\cdot,v_m\rangle_{\V_m}$ for $v_m\in\V_m$ and $\beta_n=\langle\cdot,w_n\rangle_{\V_n}$ for $w_n\in\V_n$, then
$$K(\alpha_m,\beta_n)=\langle Q(m,n)w_n,v_m\rangle_{\V_m}.$$
If $M$ is connected, then the total space $\bV^*$ is also connected 
and if $D$ is dense in an open subset of $M$, then $\coprod_{m\in D}\V^*_m$ is dense in a open subset of $\V^*$. The first assertion 
now follows from general facts about reproducing kernel spaces 
\cite{neeb00}*{Lemma~I.1.5} and the second by the the discussion in 
\cite{neeb00}*{Lemma~A.III.8}. 
\end{proof}


\begin{lem}\label{equivkernel}
If $(\pi,\H)$ is realized in $\Gamma(\V)$, then the kernel $Q$ of $\H$ satisfies 
\[ Q(g.m,g.n)=(\gamma_g|_{\V_m})Q(m,n)(\gamma_g|_{\V_n})^{-1} 
\quad \mbox{ for } \quad m,n\in M, g\in G.\]
In particular the hermitian operators 
$Q(m,m)$ commute with $\rho_m(G_m))$ for every $m\in M$. 
\end{lem}

\begin{proof}
Since 
$$(\pi(g)^{-1}s)(m)=\gamma_{g^{-1}}(s(g.m))\quad\mbox{for}\quad s\in \H,g\in G,m\in M$$ 
we have 
$$\Ev_{g.m}=\gamma_g|_{\V_m}\circ \Ev_m\circ\pi(g)^{-1}.$$
Therefore  
\begin{align*}
Q(g.m,g.n)&=\Ev_{g.m}\Ev_{g.n}^*
=((\gamma_g|_{\V_m}) \Ev_m\pi(g)^{-1})(\pi(g)\Ev_n^*(\gamma_g|_{\V_n})^{-1})\\
&=(\gamma_g|_{\V_m})Q(m,n)(\gamma_g|_{\V_n})^{-1}.
\qedhere
\end{align*}
\end{proof}

\section{Complex geometry and multiplicity-free property}\label{compgeomultfree}

In this section we prove the propagation of the multiplicity-freeness 
from the isotropy representations to the representation 
on Hilbert spaces of holomorphic sections of equivariant holomorphic vector bundles.  
This result is proved through the construction of an 
anti-unitary operator $J$ on the representation space which implements a conjugation in the commutant of the representation, from which the multiplicity-free property of the representation follows. Then we introduce the concept of $(S,\sigma^M)$-weakly visible action to prove a second version of the Propagation Theorem where the conditions are imposed on a slice of the group action on a dense in an open subset of the base space.

\subsection{Propagation of the multiplicity free property from fibers to sections}

The following lemma captures the key idea that is mostly used to 
show that a commutant is commutative. 
We refer to \cite{FT99} for one of the first systematic applications of this idea. 

\begin{lem}\label{conjcommute}
Let $\M\subseteq B(\H)$ be a von Neumann algebra. The following conditions are equivalent:
\begin{itemize}
\item[\rm(a)] $\M$ is commutative.
\item[\rm(b)] There is an anti-unitary $J$ on $\H$ such that $JAJ^{-1}=A^*$ for $A\in \M$.
\item[\rm(c)] There is an anti-unitary $J$ on $\H$ which commutes with the 
self-adjoint part $\M_h$ of $\M$, i.e., $J\in(\M_h)'^{\R}$, where $(\cdot)'^{\R}$ denotes the real linear commutant.
\item[\rm(d)] There is an anti-unitary $J$ which commutes with the positive invertible operators in $\M$.
\end{itemize}
\end{lem}

\begin{proof} (a) $\Rightarrow$ (b): 
We decompose $\H$ into cyclic representations of $\M$. 
Hence $\H=\bigoplus_{i\in I}\H_i$ where $\H_i\cong 
L^2(X_i,\mu_i)$ for  compact spaces $X_i$ and regular 
probability measures $\mu_i$ on $X_i$ 
and $\M|_{\H_i}= L^{\infty}(X_i,\mu_i)$ acts as multiplication operators on $\H_i=L^2(X_i,\mu_i)$, see \cite{fa}*{Thm. 11.32}. We define $J(\oplus_{i\in I}f_i)=\oplus_{i\in I}\overline{f_i}$ and the assertion follows since $\M_h|_{\H_i}$ are real valued functions on $X_i$.

(b) $\Leftrightarrow$ (c): Since the map 
$\M \to \M, A\mapsto JAJ^{-1}$ is 
anti-linear, it coincides with the antilinear map 
$A \mapsto A^*$ if and only if it does on the subspace 
$\M_h$ of hermitian elements. 

(b) $\Rightarrow$ (a): For $A,B\in \M$ we have
$$AB=J^{-1}(AB)^*J=J^{-1}B^*JJ^{-1}A^*J=BA,$$
so that $\M$ is commutative.

(c) $\Rightarrow$ (d):  This is trivial.

(d) $\Rightarrow$ (c): Assume that $A\in \M_h$ and choose $c\in\R$ such that $B=A+c\id$ is positive and invertible. Since $B=J^{-1}BJ$ we obtain 
\begin{align*}
A&=B-c\id=J^{-1}BJ-c\id=J^{-1}(B-c\id)J=J^{-1}AJ.
\qedhere
\end{align*}
\end{proof}

\begin{thm}\textsc{(First Propagation Theorem)}\label{multfree1}
Let a group $G$ act by automorphisms on the holomorphic Hilbert bundle $q:\V\to M$. Assume that there exists an anti-holomorphic bundle endomorphism $(\sigma^{\V},\sigma^M)$ 
and a $G$-invariant subset $D\subset M$ with $\overline{D}^\mathrm{o}\neq \emptyset$, such that for any $m\in D$ there is $g\in G$ such that $g.m=\sigma^M(m)$, and $(\gamma_g|_{\V_m})^{-1}\sigma^{\V}_m$ commutes with the hermitian part $\rho(G_m)'_h$ 
of $\rho(G_m)'$. Then any unitary representation $(\pi,\H)$ of $G$ realized in $\Gamma(\V)$ is multiplicity-free. 
\end{thm}

Note that the group $G$ is not assumed to be a Banach--Lie group. Also note that the existence of an anti-unitary operator $(\gamma_g|_{\V_m})^{-1}\sigma^{\V}_m$ commuting with $\rho_m(G_m)'_h$ implies by Lemma \ref{conjcommute} that the representation $\rho_m:G_m\to\U(\V_m)$ is multiplicity-free.

\begin{proof}
\textbf{First step.} We define a conjugate linear map
\begin{align}
J:\Gamma( \V)\to \Gamma(\V),\quad s\mapsto (\sigma^{\V})^{-1}\circ s\circ \sigma^M.\label{conj}
\end{align}
We will prove that $J:\H\to\H$ is a anti-unitary operator for any unitary representation $(\pi,\H)$ realized in $\Gamma (\V)$.

Consider the Hilbert space $\widetilde{\H}:=J(\H)\subseteq \Gamma(\V)$ equipped with the inner product
$$\langle Js_1,Js_2\rangle_{\widetilde{\H}}:=\langle s_2,s_1\rangle_{\H} \quad\mbox{for}\quad s_1,s_2\in\H,$$
so that $J:\H\to \widetilde{\H}$ is an anti-unitary operator.
Denote by $\widetilde{\Ev}_m:\widetilde{\H}\to \V_m$ the evaluations of $\widetilde{\H}$. Since 
$$\widetilde{\Ev}_m(Js)=(\sigma^{\V}_m)^{-1}(s(\sigma^M(m)))\quad\mbox{for}\quad s\in \H, m\in M,$$
we get 
$$\widetilde{\Ev}_m=(\sigma^{\V}_m)^{-1}\circ\Ev_{\sigma^M(m)}\circ J^{-1},$$
so that
\begin{align*} 
Q_{\widetilde{\H}}(m,n)&=\widetilde{\Ev}_m\widetilde{\Ev}_n^*=((\sigma^{\V}_m)^{-1}\Ev_{\sigma^M(m)} J^{-1})(J\Ev_{\sigma^M(n)}^*\sigma^{\V}_n)\\
&=(\sigma^{\V}_m)^{-1}Q_{\H}(\sigma^M(m),\sigma^M(n))\sigma^{\V}_n.
\end{align*}
We fix $m\in D$. By assumption there is $g\in G$ such that $\sigma^M(m)=g.m$ and $(\gamma_g|_{\V_m})^{-1}\sigma^{\V}_m$ commutes with the hermitian part 
$\rho(G^1_m)'_h$ of $\rho(G^1_m)'$. 
Since $Q(m,m)\in \rho_m(G_m)'_h$ by Lemma \ref{equivkernel}, 
\begin{align*}
Q_{\widetilde{\H}}(m,m)
&=(\sigma^{\V}_m)^{-1}Q_{\H}(\sigma^M(m),\sigma^M(m))\sigma^{\V}_m\\
&=(\sigma^{\V}_m)^{-1}Q_{\H}(g.m,g.m)\sigma^{\V}_m\\
&=(\sigma^{\V}_m)^{-1}(\gamma_g|_{\V_m})Q_{\H}(m,m)(\gamma_g|_{\V_m})^{-1}\sigma^{\V}_m=Q_{\H}(m,m),
\end{align*}
so that $Q_{\widetilde{\H}}=Q_{\H}$ on $\diag(D)\subseteq M\times M$. By Lemma \ref{lemeqkern}, the Hilbert space $\widetilde{\H}$ coincides with $\H$ and
\begin{equation}
  \label{eq:j-rel}
\langle Js_1,Js_2\rangle_{\H}=\langle Js_1,Js_2\rangle_{\widetilde{\H}}=\langle s_2,s_1\rangle_{\H} \quad\mbox{for}\quad s_1,s_2\in \H.
\end{equation}

\textbf{Second step}. Assume that $A\in \pi(G)'$ is positive and invertible. We define a compatible inner product on $\H$ by  
$$\langle s_1,s_2\rangle_{\H_A}:=\langle As_1,s_2\rangle_{\H}.$$
The space $\H$ with the new inner product will be denoted by $\H_A$.
For $s_1,s_2\in \H$ and $g\in G$ we have 
\begin{align*}
\langle \pi(g)s_1,\pi(g)s_2\rangle_{\H_A}&=\langle A\pi(g)s_1,\pi(g)s_2\rangle_{\H}
=\langle \pi(g)As_1,\pi(g)s_2\rangle_{\H}=\langle As_1,s_2\rangle_{\H}=\langle s_1,s_2\rangle_{\H_A}.
\end{align*}
Therefore $\pi$ also defines a unitary representation on $\H_A$ and we can 
apply \eqref{eq:j-rel} to $\H_A$ to obtain
\begin{align*}
\langle As_1,s_2\rangle_{\H}&=\langle s_1,s_2\rangle_{\H_A}=\langle Js_2,Js_1\rangle_{\H_A}=\langle AJs_2,Js_1\rangle_{\H}=\langle Js_2,AJs_1\rangle_{\H}\\
&=\langle Js_2,JJ^{-1}AJs_1\rangle_{\H}=\langle J^{-1}AJs_1,s_2\rangle_{\H}.
\end{align*}
Hence $A=J^{-1}AJ$, i.e., $J$ commutes with $A$, and by Lemma \ref{conjcommute} the von Neumann algebra $\pi(G)'$ is commutative.
\end{proof}

\subsection{Discretely decomposable representation on the fiber}

The following lemma shows how the commuting condition in 
the First Propagation Theorem can be expressed in the classical 
context where the isotropy representation decomposes discretely. 

\begin{lem}\label{conjcommutedisc}
Let $\M\subseteq B(\H)$ be a commutative von Neumann algebra, so that $\H=\oplus_{i\in I}\H_i$ and the action of $\M'$ on $\H_i$ is irreducible for each $i\in I$. Then the following conditions on an anti-unitary $J$ acting on $\H$ are equivalent:
\begin{itemize}
\item[\rm(a)] $J$ commutes with $\M_h$.
\item[\rm(b)] $J(\H_i)\subseteq\H_i$ for $i\in I$. 
\end{itemize}
\end{lem}

\begin{proof} As 
$\M_h=\bigoplus^{\infty}_{i\in I}\R \id_{\H_i}$  by Schur's Lemma, 
$J\in(\M_h)'^{\R}$ is equivalent to $J(\H_i)\subseteq \H_i$ for $i\in I$.
\end{proof}

 
\begin{prop}\label{doublecom1}
If, for $m\in M$, the isotropy representation on the fiber $\V_m$ restricted to a subgroup $G^1_m\subseteq G_m$ is multiplicity free with irreducible decomposition $\V_m=\bigoplus_{i\in I} \V_m^{(i)}$ and there exists a 
$g\in G$ such that $(\gamma_g|_{\V_m})^{-1}\sigma^{\V}_m(\V_m^{(i)})= \V_m^{(i)}$ for all $i\in I$, then the anti-unitary operator $(\gamma_g|_{\V_m})^{-1}\sigma^{\V}_m$ commutes with $\rho_m(G_m)'_h$.
\end{prop}

\begin{proof}
We apply Lemma \ref{conjcommutedisc} with $\M=\rho_m(G^1_m)'$, $J=(\gamma_g|_{\V_m})^{-1}\sigma^{\V}_m$ and $\H_i=\V_m^{(i)}$ and conclude that the anti-unitary operator $(\gamma_g|_{\V_m})^{-1}\sigma^{\V}_m$ commutes with $\rho_m(G^1_m)'_h$, and hence also with $\rho_m(G_m)'_h$. 
\end{proof}

Using Proposition \ref{doublecom1}, we obtain a special case of Theorem \ref{multfree1} which is an infinite dimensional version of \cite{kob}*{Theorem 2.2}.

\begin{thm}\label{multfree2}
Let a group $G$ act by automorphisms on the holomorphic Hilbert bundle $q:\V\to M$. Assume that there exists an anti-holomorphic bundle endomorphism $(\sigma^{\V},\sigma^M)$ and a $G$-invariant subset $D$ which is dense in an open subset of $M$ such that for any $m\in M$:
\begin{itemize}
\item[\rm{(F)}] The isotropy representation on the fiber $\V_m$ restricted to a subgroup $G^1_m\subseteq G_m$ is multiplicity free with irreducible decomposition 
$\V_m=\bigoplus_{i\in I} \V_m^{(i)}$.
\item[\rm{(C)}] There exists $g\in G$ such that $(\gamma_g|_{\V_m})^{-1}\sigma^{\V}_m(\V_m^{(i)})= \V_m^{(i)}$ for all $i\in I$. 
\end{itemize}
Then any unitary representation $(\pi,\H)$ of $G$ realized in $\Gamma(\V)$ is multiplicity-free. 
\end{thm}

\subsection{Weakly visible actions on complex manifolds}

\begin{defn}\label{defweakvis}
\rm{If a group $G$ acts by holomorphic maps on a connected complex manifold $M$ we say that the action is \textit{$(S,\sigma^M)$-weakly visible} if $S$ is a 
subset $M$ and $\sigma^M$ is an antiholomorphic diffeomorphism of $M$ satisfying  
\begin{itemize}
\item[\rm{(WV1)}] \label{visible1} $D=G.S$ is dense in an open subset of $M$.  
\item[\rm{(WV2)}] \label{visible2} $\sigma^M|_S=\id$. 
\item[\rm{(WV3)}] \label{visible3} $\sigma^M$ preserves every $G$-orbit in $D$.
\end{itemize}
}
\end{defn}

\begin{rem}\label{section}
\rm{Since $S$ meets every $G$-orbit in $D$, 
there exists a subset $S_0 \subset S$ which 
 meets every $G$-orbit in a single point. Then $S_0$ also satisfies (WV1-3).}
\end{rem}

\begin{rem}
\rm{Condition (WV3) in Definition \ref{defweakvis} follows from conditions (WV1/2) if there is an automorphism $\sigma^G$ of $G$ such that
\begin{align*}
\sigma^G (g).\sigma^M(m)=\sigma^M(g.m) \quad\mbox{for}\quad g\in G,m\in D. 
\end{align*}  
}
\end{rem}

\begin{defn}
\rm{A bundle endomorphism $(\sigma^{\V},\sigma^M)$ is said to be \textit{compatible with the $G$-action and the automorphism $\sigma^G$ of $G$} if $\sigma^{\V}$ is an intertwining map between the action $\gamma$ and the action $\gamma\circ\sigma^G$, i.e., on $\bV$ 
we have 
\begin{align}
\gamma_{\sigma^G(g)}=\sigma^{\V}\circ\gamma_g\circ(\sigma^{\V})^{-1}
\quad \mbox{ for } \quad g \in G. \label{resp}
\end{align}
This implies in particular 
\begin{align}
\sigma^G(g).\sigma^M(m)=\sigma^M(g.m) \quad\mbox{for}\quad g\in G,m\in M.
\label{compat}
\end{align} }
\end{defn}

\begin{prop}
If the anti-holomorphic bundle endomorphism $(\sigma^{\V},\sigma^M)$ is compatible with the $G$-action and with an automorphism $\sigma^G$ of $G$, then the representation $\delta$ of $G$ on $\Gamma(\V)$ satisfies 
$$\delta_{\sigma^G(g)}=J^{-1}\delta_gJ,\quad\mbox{for}\quad g\in G,$$
where $J$ is the anti-linear operator defined in \eqref{conj}.
\end{prop}

\begin{proof}
This is immediate from \eqref{actionbundle}, \eqref{conj} and \eqref{resp}. 
\end{proof}

\subsection{Conditions on a slice $S$}

By using the concept of a
 weakly-visible action we give a second form of the propagation theorem where the conditions are imposed on the slice $S$ instead on all of $D$.

\begin{thm}\textsc{(Second Propagation Theorem).}\label{multfreevis}
Let $q:\V\to M$ be a $G$-equivariant holomorphic vector bundle with an anti-holomorphic vector bundle endomorphim $(\sigma^{\V},\sigma^M)$ which is compatible 
with the $G$-action and the automorphism $\sigma^G$ of $G$.
Assume also that
\begin{itemize}
\item[\rm{(B)}] \label{cond21} The action on $M$ is $(S,\sigma^M)$-weakly visible.
\item[\rm{(F)}] \label{cond22} For any $s\in S$ the anti-unitary operator  $\sigma^{\V}_s$ commutes with $\rho_s(G_s)'_h$. 
\end{itemize}
Then any unitary representation $(\pi,\H)$ of $G$ realized in $\Gamma (\V)$ is multiplicity-free.
\end{thm}

\begin{proof}
We set $D:=G.S$ which is dense in an open subset of $M$ since the action on the base space is $(S,\sigma^M)$-weakly visible. By Remark \ref{section} we may assume that each $G$-orbit intersects $S$ only once. Then for $m\in D$ we get $m=g.s$ for some 
$g\in G$ and a unique $s\in S$. 

If we define $c_g:G_s\xrightarrow{\sim} G_m$, $\ell\mapsto g\ell g^{-1}$, then $\gamma_g|_{\V_s}:\V_s\to \V_m$ satisfies 
$$
\xy
\xymatrix{
\V_s \ar[d]_{\rho_s(\ell)} \ar[r]^{\gamma_g} & \V_m \ar[d]^{\rho_m(c_g(\ell))}
\\
\V_s \ar[r]^{\gamma_g} & \V_m
}\endxy
$$
for $\ell\in G_s$. Hence, since $G_m=c_g(G_s)$ we get
\begin{align}
((\rho_m(G_m))'_h)'^{\R}=(\gamma_g|_{\V_s})((\rho_s(G_s))'_h)'^{\R}(\gamma_g|_{\V_s})^{-1}.
\label{isomcom}
\end{align}
For $g':=\sigma^G(g)g^{-1}\in G$ we have
\begin{align*}
\sigma^M(m)&=\sigma^M(g.s)=\sigma^G(g).\sigma^M(s)&\mbox{by \eqref{compat}}\\
&=\sigma^G(g).s&\mbox{by (WV2) in Def. \ref{defweakvis}}\\
&=\sigma^G(g)g^{-1}g.s=g'.m.
\end{align*}
Using \eqref{resp} we obtain
\begin{align*}
(\gamma_{g'}|_{\V_m})^{-1}\sigma_m^{\V}=(\gamma_g|_{\V_s})(\gamma_{\sigma^G(g^{-1})}|_{\V_{\sigma^M(m)}})\sigma_m^{\V}
=(\gamma_g|_{\V_s})\sigma_s^{\V}(\gamma_g|_{\V_s})^{-1}.
\end{align*}
Since $\sigma^{\V}_s$ commutes with $\rho_s(G_s)'_h$ we conclude with \eqref{isomcom} that $(\gamma_{g'}|_{\V_m})^{-1}\sigma_m^{\V}$ commutes with $\rho_s(G_m)'_h$. Hence all the assumptions of Theorem \ref{multfree1} hold and the conclusion follows.
\end{proof}

\begin{rem}
\rm{Choosing $g=e$ for $m\in S$ in  Proposition \ref{doublecom1} we can replace condition (F) in Theorem \ref{multfreevis} by the weaker condition:
\begin{itemize}
\item[\rm(F')] For $s\in S$, the isotropy representation on the fiber $\V_s$ restricted to a subgroup $G^1_s\subseteq G_s$ is multiplicity-free with irreducible decomposition $\V_s=\bigoplus_{i\in I} \V_s^{(i)}$ and $\sigma^{\V}_s(\V_s^{(i)})=\V_s^{(i)}$ for any $i\in I$. 
\end{itemize}
We thus obtain a special case of Theorem \ref{multfreevis}}
\end{rem}

\section{Associated bundles}\label{assocbundles}

In this section we recall from \cite{bel} and \cite{neeb13a} how to define complex structures on associated Banach bundles. 
Based on these complex structures, we prove that certain vector bundle endomorphisms are antiholomorphic. 

Let $G$ be a Banach--Lie group with Lie algebra $\g$ and $H\subseteq G$ be a \textit{split Lie subgroup}, i.e., the Lie algebra $\h$ of $H$ has a closed complement in $\g$. Hence the homogeneous space $M:=G/H$ has a 
smooth manifold structure such that the projection 
$q_M:G\to G/H, g \mapsto gH$ 
is a submersion and defines a smooth $H$-principal bundle. 

Let $q:\V=G\times_H V\to M$ be a homogeneous vector bundle defined by the norm continuous representation $\rho:H\to \GL(V)$ on a Banach space $V$. We associate to each section $s:M\to \V$ the function $\tilde{s}:G\to V$ specified by $s(gH)=[g,\tilde{s}(g)]$ for 
$g\in G$. Then a 
function $f:G\to V$ is of the form $\tilde{s}$ for a section $s$ of $\V$ if and only if 
\begin{align}\label{funsec}
f(gh)=\rho(h)^{-1}f(g)\quad\mbox{for}\quad g\in G,h\in H. 
\end{align}
We write $C^{\infty}(G,V)_\rho$ for the set of 
smooth functions $f:G\to V$ satisfying \eqref{funsec}.

For every $g\in G$, we then have isomorphisms 
\[ \iota_g:V\to \V_{gH}=[g,V],\ v\mapsto [g,v]\] 
and the group $G$ acts in a canonical way on $\V\to M$ by bundle automorphism $\gamma_g([g',v])=[gg',v]$ and $g.g'H=gg'H$ for $g,g'\in G$ and $v\in V$.

\subsection{Complex structures on associated bundles}\label{compstructassoc}

We assume that the coset space $M:=G/H$ carries the structure of a complex manifold such that $G$ acts on $M$ by biholomorphic maps. Let $m_0:=q_M(e)\in M$ 
be the canonical base point and $\q\subseteq \g_{\C}$ be the kernel of the complex linear extension of the map $\g\to T_{m_0}(G/H)$, so that $\q$ is a complex linear subalgebra of $\g_{\C}$ invariant under $\Ad_H$. We call $\q$ the \textit{subalgebra defining the complex structure} on $M=G/H$ because specifying $\q$ means to identify $T_{m_0}(G/H)\simeq \g/\h$ with the complex Banach space $\g_{\C}/\q$, and thus specifying the complex structure on $M$. The subalgebra $\q$ satisfies $\q+\bar{\q}=\g_{\C}$, $\q\cap\bar{\q}=\h_{\C}$ and $\Ad_H(\q)=\q$. See \cite{bel} for further information on complex structures on homogeneous Banach manifolds.

We want to define complex structures on 
vector bundles $\V=G\times_H V$ over Banach homogeneous spaces $M=G/H$ associated to a norm continuous representation $\rho:H\to \GL(V)$ of the isotropy group on a complex Banach space $V$. Suppose that $H\subseteq G$ is a split 
Lie subgroup and, as above, $\q\subseteq \g_{\C}$ is a closed complex 
$\Ad(H)$-invariant subalgebra containing $\h_{\C}$ and specifying the complex 
structure on $G/H$. 
If $\rho:G\to\GL(V)$ is a norm continuous representation on the Banach space $V$, then a morphism $\beta:\q\to\gl(V)$ of complex Banach--Lie algebras is said to be an \textit{extension} of $\dd\rho$ if 
\begin{align}
\dd\rho=\beta|_{\h}\quad\mbox{and}\quad\beta(\Ad(h)x)=\rho(h)\beta(x)\rho(h)^{-1} \quad\mbox{for}\quad h\in H ,x\in \q .
\label{defext}
\end{align}

We associate to each $x\in \g$ the left invariant differential operator $L_x$ on $C^{\infty}(G,V)$ given by
$$(L_xf)(g)=\frac{d}{dt}\biggl|_{t=0}f(g\exp(tx))$$
and by complex linear extension we define the operators
\begin{align}
L_{x+iy}:=L_x+iL_y\quad\mbox{  for   }\quad  x,y\in \g.
\label{defdifop}
\end{align}   
For any extension $\beta$ of $\rho$, we write $C^{\infty}(G,V)_{\rho,\beta}$ for the subspace of those elements of $C^{\infty}(G,V)_\rho$ satisfying 
\begin{align}
L_wf=-\beta(w)f\quad \mbox{ for }\quad w\in\q. 
\label{defholsec}
\end{align} 

We recall the following theorem from \cite{neeb13a}*{Thm.~1.6}.

\begin{thm}
Let $V$ be a complex Banach space and $\rho:H\to \GL(V)$ be a norm continuous representation. Then, for any extension $\beta:\q\to\gl(V)$ of $\rho$, the associated bundle $\V=G\times_H V$ carries a unique structure of a holomorphic vector bundle over 
$M= G/H$, which is determined by the characterization of holomorphic sections $s:M\to \V$ as those for which $\tilde{s}\in C^{\infty}(G,V)_{\rho,\beta}$. Any such holomorphic bundle structure is $G$-invariant in the sense that $G$ acts on $\V$ by holomorphic vector bundle automorphisms. Conversely, every $G$-invariant holomorphic vector bundle structure on $\V$ is obtained from this construction.
\end{thm}

\subsection{The endomorphism bundle}\label{end}

Let $\V=G\times_H V\to M$ be a $G$-homogeneous Hilbert 
bundle as in Theorem \ref{compstructassoc}, where the representation 
$\rho$ is unitary. 
Then the complex manifold $M\times  M^{\rm op}$ is a complex homogeneous space $(G\times G)/(H\times H)$, where the complex structure is defined by the closed subalgebra $\q\oplus \bar{\q}$ of $\g_\C\oplus\g_\C$. On the Banach space $B(V)$ we consider the norm continuous representation 
\[ \tilde{\rho}:H\times H\to \GL(B(V))\quad \mbox{  given by } \quad 
\tilde{\rho}(h_1,h_2)A=\rho(h_1)A\rho(h_2)^*  = \rho(h_1)A\rho(h_2)^{-1} \] 
and the corresponding extension $\tilde{\beta}:\q\oplus\bar{\q}\to\gl(B(V))$ by $$\tilde{\beta}(x_1,x_2)A=\beta(x_1)A+A\beta(\overline{x_2})^*.$$
We write $\L:=(G\times G)\times_{(H\times H)}B(V)$ for the corresponding holomorphic Banach bundle over $M\times M^{\rm op}$. 
For every pair $(g_1,g_2)\in G\times G$ we have an isomorphism
\[ \gamma_{(g_1,g_2)} :B(V)\to B(\V_{q_M(g_2)},\V_{q_M(g_1)}),\quad \gamma_{(g_1,g_2)}(A)[g_2,v]= [g_1,Av].\]
This defines a map
\[ \gamma :G\times G\times B(V)\to B(\V)=\coprod_{m,n\in M} B(\V_m,\V_n),\quad \gamma (g_1,g_2,A)=\gamma_{(g_1,g_2)}(A).\]
For $h_1, h_2\in H$, we have
\begin{align*}
\gamma(g_1h_1,g_2h_2,\tilde{\rho}(h_1,h_2)^{-1}A)[g_2,v]&=\gamma(g_1h_1,g_2h_2,\tilde{\rho}(h_1,h_2)^{-1}A)[g_2h_2,\rho(h_2)^{-1}v]
\\
&=[g_1h_1,\rho(h_1)Av]=[g_1,Av]=\gamma(g_1,g_2,A)[g_2,v]
\end{align*} 
so that $\gamma$ factors through a bijection $\bar{\gamma}: \L\to B(\V)$. 
This provides a description of the bundle $\L$ as the endomorphism bundle of the Hilbert 
bundle $\V$, see \cite{br07}.

\subsection{Holomorphic morphism of equivariant principal bundles}

Let $q_i:\V_i=G_i\times_{H_i} V_i\to M_i=G_i/H_i$ for $i=1,2$ be homogeneous vector bundles defined by norm continuous unitary representations $\rho_i:H_i\to \U(V_i)$ on Hilbert spaces $V_i$.

Let $\lambda:G_1\to G_2$ 
be a homomorphism of Banach--Lie groups satisfying 
$\lambda(H_1)\subseteq H_2$, so that there is an induced map $\lambda^M:M_1\to M_2$ defined by $gH_1\mapsto \lambda(g)H_2$.  
Let $\psi:V_1\to V_2$ be an operator such that  
\begin{align}
\rho_2(\lambda(h))\circ\psi=\psi\circ\rho_1 (h) \quad\mbox{for}\quad h\in H_1.\label{intertwin}
\end{align}
We say that $(\lambda,\psi)$ is a morphism between the representations $\rho_1$ and $\rho_2$. These conditions imply that the map  
$$\lambda^{\V}:\V_1\to \V_2,\quad [g,v]\mapsto [\lambda (g),\psi (v)]$$
is well defined. It is a lift of the map $\lambda^M$ and it is complex linear on each fiber. If we differentiate \eqref{intertwin} we obtain
\begin{align*}
\dd\rho_2(\dd\lambda(w))\circ\psi=\psi\circ\dd\rho_1 (w) \quad\mbox{for}\quad w\in \h_1.
\end{align*}
If the $M_i=G_i/H_i$ carry complex structures defined by subalgebras $\q_i$, then $\lambda^M$ is holomorphic if and only if the complex linear extension $\dd\lambda_{\C}:\g_{1,\C}\to\g_{2,\C}$ satisfies $\dd\lambda_{\C}(\q_1)\subseteq \q_2$.
Assume further that $q_i:\V_i\to M_i$ have complex bundle structures defined by extensions $\beta_i:\q_i\to \gl(V_i)$ of $\dd\rho_i$ for $i=1,2$. If the intertwining operator $\psi$ also satisfies 
\begin{align}
\beta_2(\dd\lambda(w))\circ\psi=\psi\circ\beta_1 (w) \quad\mbox{for}\quad w\in \q_1,\label{intertwindif}
\end{align}
then we say that $(\lambda,\psi)$ is a \textit{morphism} between 
the representation $(\rho_1,\beta_1)$ of $(H,\fq)$ 
on $V_1$ and the representation $(\rho_2,\beta_2)$ on $V_2$.

\begin{prop}\label{holoend}
Assume that $(\lambda,\psi)$ is an intertwiner 
for the $H$-representations 
$(\rho_1, V_1)$ and $(\rho_2, V_2)$. 
Then $(\lambda,\psi)$ is a morphism between the representation $(\rho_1,\beta_1)$ of $V_1$ and the representation $(\rho_2,\beta_2)$ of $V_2$ if and only if $\lambda^M:M_1\to M_2$, $gH_1\mapsto \lambda(g)H_2$ and $\lambda^{\V}:\V_1\to \V_2$ are holomorphic.
\end{prop} 

\begin{proof} The condition 
$\dd\lambda_{\C}(\q_1)\subseteq \q_2$ holds 
if and only if the induced map $\lambda^M:M_1\to M_2$, $gH_1\mapsto \lambda(g)H_2$ is holomorphic. Assume that this is the case. 
We represent $\lambda^{\V}$ in local trivializations 
\[ \alpha_i:U_i\times V_i\to \V_i|_{U_i},\quad (gH_i,v)\mapsto 
[g,F_i(g)v]\quad\mbox{for}\quad i=1,2,\] 
where $U_i\subseteq M_i$ are open subsets and
$$F_i:q_M^{-1}(U_i)\to \GL(V_i)$$ satisfies 
\begin{align}
&F_i(gh)=\rho_i(h)^{-1}F_i(g) \quad\mbox{for}\quad g\in q_M^{-1}(U_i),h\in H_i\nonumber\\
&L_wF_i=-\beta_i (w)F_i \quad\mbox{for}\quad w\in \q_i\label{difext1}
\end{align}
(see the proof of \cite{neeb13a}*{Thm 1.6}).  Observe that 
$$
\xy
\xymatrix{
(gH_1,v) \ar@{|->}[d] \ar@{|->}[r] & [g,F_1(g)v]\ar@{|->}[d]^{\lambda^{\V}}
\\
(\lambda(g)H_2, F_2(\lambda (g))^{-1}\psi F_1(g)v)  & [\lambda(g),\psi(F_1(g)v)]\ar@{|->}[l]
}\endxy
$$
so that in these trivializations the map is given by
\[ (gH_1,v)\mapsto (\lambda(g)H_2, F_2(\lambda (g))^{-1}\psi F_1(g)v)
\quad\mbox{for}\quad gH_1\in U_1\cap (\lambda^M)^{-1}(U_2),v\in V.\]
Since the map $F_2:q_M^{-1}(U_2)\to \GL(V_2)$ satisfies \eqref{difext1} we have 
\begin{align}\label{adifext2}
L_w(F_2\circ\lambda)=L_{\dd\lambda(w)}(F_2)\circ\lambda=-\beta_2 (\dd\lambda(w))(F_2\circ\lambda).
\end{align}
The map $gH_1\mapsto \lambda(g)H_2$ is holomorphic and the map 
$F_2(\lambda (g))^{-1}\psi F_1(g) \: V_1\to V_2$ 
is complex linear. We have to prove that the map $gH_1\mapsto F_2(\lambda (g))^{-1}\psi F_1(g)\in B(V_1,V_2)$ for $gH_1\in U_1\cap (\lambda^M)^{-1}(U_2)$ is holomorphic, i.e., that the map $g\mapsto F_2(\lambda (g))^{-1}\psi F_1(g)$ is annihilated by the differential operators $L_w$ for $w\in\q$, 
if and only if \eqref{intertwindif} holds. Observe that 
\begin{align*}
L_w((F_2\circ\lambda)^{-1} (\psi F_1))=&L_w((F_2\circ\lambda)^{-1})(\psi F_1)+(F_2\circ\lambda)^{-1}L_w(\psi F_1)\\
=&-(F_2\circ\lambda)^{-1}L_w(F_2\circ\lambda)(F_2\circ\lambda)^{-1}(\psi F_1)+(F_2\circ\lambda)^{-1}L_w(\psi F_1)\\
=&-(F_2\circ\lambda)^{-1}(-\beta_2(\dd \lambda(w)))(F_2\circ\lambda)(F_2\circ\lambda)^{-1}(\psi F_1)\\
&+(F_2\circ\lambda)^{-1}(\psi L_w F_1)\qquad\qquad \qquad\qquad\qquad \qquad 
\mbox{ by \eqref{adifext2}}\\
=&-(F_2\circ\lambda)^{-1}(-\beta_2(\dd \lambda(w))(\psi F_1))\\
&+(F_2\circ\lambda)^{-1}\psi (-\beta_1(w) F_1)
\qquad \qquad\qquad\qquad \qquad \mbox{ by \eqref{difext1}}\\
=&(F_2\circ\lambda)^{-1}(\beta_2(\dd \lambda(w))\psi -\psi (\beta_1(w)))F_1.
\end{align*}
Therefore $L_w((F_2\circ\lambda)^{-1} (\psi F_1))=0$ if and only if $\beta_2(\dd \lambda(w))\psi -\psi (\beta_1(w))=0$ for ${w\in \q_1}$. 
\end{proof}

\begin{rem}
\rm{It is easy to check that the correspondence between morphisms of representations and holomorphic bundle morphisms is functorial. In particular one has for a 
smooth Banach $G$-module $W$ the following linear 
bijection, which is  a variant of Frobenius reciprocity: 
  \begin{equation}
    \label{eq:frob}
\Gamma \: \Hom_G(W,C^{\infty}_{\rho,\beta}(G,V))\to \Hom_{\rho,\beta}(W,V),\quad 
\Phi\mapsto {\rm ev}_e \circ \Phi.
  \end{equation}
Here the intertwining operators are assumed to be continuous and the space 
$C^{\infty}_{\rho,\beta}(G,V)$ carries the locally convex topology of pointwise 
convergence. Then the inverse of $\Gamma$ is simply given by 
$\Gamma^{-1}(\varphi)(w)(g)=\varphi(\pi(g^{-1})w)$.
}
\end{rem} 

\subsection{Opposite complex structure on associated bundles}\label{oppositebundle}

Given an associated vector bundle $q:\V=G\times_{H, \rho}V\to M=G/H$ endowed with a complex structure defined in Subsection \ref{compstructassoc} by a subalgebra $\q\subseteq \g_{\C}$ and an extension $\beta$ of $\dd\rho$, we can define a bundle 
$q:\V^{\rm op}\to  M^{\rm op}$ with the same underlying real spaces and opposite complex structures on the total and base spaces. 

\begin{prop}\label{propoppositebundle}
If $\bar{\rho}:H\to\U(V^{\rm op})$ is the complex conjugate representation of $\rho$ then the map $\bar{\beta}:\bar{\q}\to \gl(V^{\rm op})$  given by $\bar{\beta}(w):=\overline{\beta(\overline{w})}$ for $w\in\bar{\q}$ is an extension of of $\dd\bar{\rho}$. The bundle 
$q:\V^{\rm op}=G\times_{H,\bar{\rho}}V^{\rm op}\to  M^{\rm op}=G/H$ with the complex structure defined by the subalgebra $\bar{\q}\subseteq \g_{\C}$ and the extension $\bar{\beta}$ of $\dd\bar{\rho}$ has
the same holomorphic sections as $q:\V=G\times_{H,\rho}V\to M=G/H$, so it carries 
the opposite complex structure on total and base space, respectively.  
\end{prop}

\begin{proof}
The algebra $\bar{\q}$ defines the opposite complex structure on $M=G/H$. Since $\bar{\beta}$ is a complex linear Lie algebra homomorphism and since $$\dd\bar{\rho}(w)=\overline{\dd\rho(w)}=\overline{\beta(w)}=\overline{\beta(\overline{w})}=\bar{\beta}(w)\quad\mbox{for}\quad w\in \h$$  
we get $\dd\bar{\rho}=\bar{\beta}|_{\h}$. Also, 
\begin{align*}
\bar{\beta}(\Ad(h)w)&=\overline{\beta(\overline{\Ad(h)w})}=\bar{\beta}(\Ad(h)\overline{w})=\bar{\rho}(h)\overline{\beta(\overline{w})}\bar{\rho}(h)^{-1}=\bar{\rho}(h)\bar{\beta}(w)\bar{\rho}(h)^{-1}
\end{align*} 
for $h\in H$ and $w\in \bar{\q}$, so 
that $\bar{\beta}:\bar{\q}\to \gl(V^{\rm op})$ is an extension of $\dd\bar{\rho}$.

For a map $f:G\to V$ we write $\bar{f}:G\to V^{\rm op}$ for the same map to $V$ with the opposite complex structure. Then
$$\bar{f}(gh)=\bar{\rho}(h)^{-1}\bar{f}(g)\quad\mbox{for}\quad g\in G,h\in H,$$
and
\[ L_w\bar{f}=\overline{L_{\overline{w}}f}
=-\overline{\beta(\overline{w})f}=-\overline{\beta(\overline{w})}\bar{f}
=-\bar{\beta}(w)\bar{f}\quad \mbox{ for }\quad w\in\bar{\q}.\]
Therefore $C^{\infty}(G,V)_{\rho,\beta}= C^{\infty}(G,V^{\rm op})_{\bar{\rho},\bar{\beta}}$, and the bundles $q:\V\to M$ and $q:\V^{\rm op}\to  M^{\rm op}$ 
have the same holomorphic sections. As both 
bundles have the same sections and $ M^{\rm op}$ and 
$V^{\rm op}$ carry the complex structure opposite to the one $M$ and $V$, 
respectively, the bundle $\V^{\rm op}$ carries the complex structure opposite to $\V$. 
\end{proof}

\subsection{Anti-holomorphic automorphisms of equivariant principal bundles}

Assume that there is an automorphism $\sigma$ of $G$ which stabilizes $H$ and such that the complex linear extension $\dd\sigma_{\C}:\g_{\C}\to\g_{\C}$ of  $\dd\sigma$ satisfies $\dd\sigma_{\C}(\q)=\bar{\q}$, i.e., the induced map $\sigma^M$ on $M=G/H$ defined by $gH\mapsto \sigma(g)H$ is anti-holomorphic. The condition 
$\bar{\rho}\circ\sigma\simeq\rho$, 
where $\bar{\rho}:H\to \U(V^{\rm op})$ is the complex conjugate representation, 
holds if and only if there is an anti-unitary operator $\psi:V\to V$ satisfying 
\begin{align}
\rho(\sigma(h))\circ\psi=\psi\circ\rho (h) \quad\mbox{for}\quad h\in H.\label{dualrep}
\end{align}
This condition implies that the endomorphism of $\V$ given by 
$$\sigma^{\V}:\V\to \V,\quad [g,v]\mapsto [\sigma (g),\psi (v)]$$
is well-defined. It is a lift of the anti-holomorphic map $\sigma^M:M\to M$, $gH\mapsto \sigma(g)H$ and it is anti-unitary on each fiber.

If the antiunitary map $\psi \: V \to V$ also satisfies 
\begin{align}
\beta(\overline{\dd\sigma(w)})\circ\psi=\psi\circ\beta (w)\label{dualrepdif}
\quad \mbox{ for } \quad w\in \q, 
\end{align}
then $(\sigma,\psi)$ is a morphism between $(\rho,\beta)$ and $(\bar{\rho},\bar{\beta})$. If we consider the endomorphism $(\sigma^{\V},\sigma^M)$ of $\V$ as a morphism 
\[ \xy
\xymatrix{
\V \ar[d]_{q} \ar[r]^{\sigma^{\V}} & \V^{\rm op} \ar[d]^{q}
\\
M \ar[r]^{\sigma^M} &  M^{\rm op}, 
}\endxy
\] 
then Propositions \ref{holoend} and \ref{propoppositebundle} lead to:

\begin{prop}\label{aholoend}
If $(\sigma,\psi)$ intertwines the representations $(\rho,\beta)$ and 
$(\bar{\rho},\bar{\beta})$, then \break 
${\sigma^M:M\to M}$, $gH\mapsto \sigma(g)H$ and $\sigma^{\V}:\V\to \V$, $[g,v]\mapsto [\sigma (g),\psi (v)]$ are anti-holomorphic.
\end{prop} 

\section{A Propagation Theorem for associated bundles}\label{propagassocbund}
Based on the constructions of the previous section we give a formulation of the propagation theorem for infinite-dimensional associated holomorphic vector bundles, which is best suited for specific examples.

Let $K$ be a subgroup of the connected Banach--Lie group $G$. For $g\in G$, the isotropy group of $m=gH$ for the action of $K$ on $M=G/H$ is $K_m=K\cap G_m=K\cap gHg^{-1}$. We define a homomorphism 
$$c_{g^{-1}}:K_{gH}\to H,\quad k\mapsto g^{-1}kg$$
with range $K_{(g)}:=c_{g^{-1}}(K_{gH})=g^{-1}Kg\cap H\subseteq H$.
If $\sigma$ is an automorphism of $G$ such that $\sigma(K)=K$ and $\sigma(H)=H$, for $g\in G^{\sigma}$ we get $\sigma(K_{(g)})=K_{(g)}$.

\begin{thm}\textsc{(Third Propagation Theorem).}\label{multfreegroup}
Let $\V=G\times_{H,\rho} V\to M$ be a $G$-equivariant holomorphic vector bundle with a complex subalgebra $\q$ and an extension \break $\beta:\q\to \gl(V)$ of $\dd\rho$ defining a complex structure on it. Let $K\subseteq G$ be a subgroup 
and $\sigma$ be an automorphism of the Banach--Lie group $G$ stabilizing $K$ and $H$ such that $(\sigma,\psi)$ is a morphism between the representations $(\rho,\beta)$ and $(\bar{\rho},\bar{\beta})$. Suppose that there is a subset $B$ of $G^{\sigma}$ such that:
\begin{itemize}
\item[\rm{(B)}]\label{cond31} The closure of the subset $KBH\subeq G$ 
has interior points. 
\item[\rm{(F)}]\label{cond32} For every $b\in B$ the anti-unitary operator $\psi$ commutes with $\rho(K_{(b)})'_h$.
\end{itemize}
Then any unitary representation $(\pi,\H)$ of $K$ 
realized in $\Gamma (\V)$ is multiplicity free. 
\end{thm}

\begin{proof}
Since $(\sigma,\psi)$ is a morphism between the representations $(\rho,\beta)$ and $(\bar{\rho},\bar{\beta})$, by Proposition \ref{aholoend} the bundle endomorphisms $(\sigma^{\V},\sigma^M)$ given by 
$$\sigma^{\V}([g,v])=[\sigma^G(g),\psi(v)]\quad\mbox{and}\quad \sigma^M(gH)
=\sigma(g)H \quad\mbox{for}\quad g\in G, v\in V$$
is antiholomorphic.

The antiholomorphic vector bundle endomorphim $(\sigma^{\V},\sigma^M)$ is compatible with the $K$-action and with the automorphism $\sigma^K:=\sigma|_K$ of $K$ since for $k\in K$, $g\in G$ and $v\in V$
\[ \sigma^M(k.gH)=\sigma^M(kg).H=\sigma^K(k).\sigma(g).H=\sigma^K(k).\sigma^M(gH),\] 
and
\begin{align*}
(\gamma_{\sigma^K(k)}\circ\sigma^{\V})([g,v])&=\gamma_{\sigma^K(k)}([\sigma(g),\psi(v)])=[\sigma^K(k)\sigma(g),\psi(v)]\\
&=([\sigma(kg),\psi(v)])=\sigma^{\V}([kg,v])=(\sigma^{\V}\circ\gamma_k)([g,v]).
\end{align*}

Since $KBH$ is dense in an open subset of $G$, 
the $K$-invariant subset $D:=KBH/H
\subseteq G/H$ is dense in an open subset of $M=G/H$. 
We set 
\[ S:=BH/H\subseteq D,\] 
so that $K.S=D$. Then $\sigma^M|_S=\id$ because $B\subseteq G^{\sigma^G}$. For $s\in S$ we have
$$\sigma^M(K.s)=\sigma^K(K).\sigma^M(s)=K.s.$$
Therefore the $K$-action on the base space is $(S,\sigma^M)$-weakly visible and 
condition (B) in Theorem \ref{multfreevis} is satisfied. 

Let $m=gH$ for $g\in G$. Via the bijection $\iota_g:V\xrightarrow{\sim}\V_{gH}$, $v\mapsto [g,v]$ and the group homomorphism $c_{g^{-1}}:K_{gH}\to K_{(g)}$ the 
isotropy representation $\rho_m:K_m\to \U(\V_m)$ factors through the representation $\rho|_{K_{(g)}}:K_{(g)}\to \U(V)$, namely  
\[ \xy
\xymatrix{
V \ar[d]_{\rho(c_{g^{-1}}(k))} \ar[r]^{\sim}_{\iota_g} & \V_m \ar[d]^{\rho_m(k)}
\\
V \ar[r]^{\sim}_{\iota_g} & \V_m, 
}\endxy
\] 
and thus 
\[ \iota_g\rho(K_{(g)})\iota_g^{-1}=\rho_s(K_m).\] 

Let $s\in S$ and $b\in B$ with $s=bH$. Observe that 
\begin{align*}
\iota_{\sigma (g)}\circ \psi=\sigma^{\V}_{gH}\circ\iota_g\quad\mbox{for}\quad g\in G,
\end{align*}
and $\sigma(b)=b$, so that $\sigma_s^{\V}=\iota_b\psi\iota_b^{-1}$. Since $\psi$ commutes with $\rho(K_{(b)})'_h$ we conclude that 
$\sigma_s^{\V}=\iota_b\psi\iota_b^{-1}$ commutes with $(\iota_b\rho(K_{(b)})\iota_b^{-1})'_h=\rho_s(K_s)'_h$
so that condition (F) in Theorem~\ref{multfreevis} holds as well.
\end{proof}

\begin{rem}\label{samesubgroup}
\rm{
Note that 
\[ N:=Z_{K\cap H}(B)=\{k\in K\cap H:kb=bk \quad\mbox{for}\quad b\in B\}\subseteq b^{-1}Kb\cap H=K_{(b)} \quad \mbox{  for } \quad b\in B.\]
 Since the commutant of $\rho(N)'_h$ is contained 
in the commutant of $\rho(K_{(b)})'_h$, 
we obtain a corollary of Theorem \ref{multfreegroup} by replacing condition (F) by: 
\begin{itemize}
\item[\rm{(F')}] The anti-unitary operator $\psi$ commutes with $\rho(N)'_h$ for 
$N=Z_{K\cap H}(B)$. 
\end{itemize}
}
\end{rem}

\begin{rem}
\rm{
In the previous formulation of the theorem, if the subset $B$ is bigger, condition (B) becomes weaker and condition (F') stronger.
}
\end{rem}

\begin{prop}\label{multfreedual}
If $\bar{\rho}\circ\sigma\simeq\rho$ as representations of $H$, with the isomorphism given by an anti-unitary $\psi$, the restriction $\rho|_{K^1}$ to a subgroup $K^1\subseteq H$ is multiplicity free with irreducible decomposition $\rho|_{K^1}\simeq\bigoplus_{i\in I}\nu^{(i)}$ on $V=\bigoplus_{i\in I} V^{(i)}$ and $\bar{\nu}^{(i)}\circ\sigma\simeq\nu^{(i)}$ as representations of $K^1$, then $\psi$ commutes with $\rho(K^1)'_h$.
\end{prop}

\begin{proof}
Since $\bar{\nu}^{(i)}\circ\sigma\simeq\nu^{(i)}$ as representations of $K^1$, there are anti-unitary $K^1$-intertwining operators $\psi_i:V^{(i)}\to V^{(i)}$ for 
$i\in I$. Then the 
unitary operator $\Psi := \psi^{-1}\circ \bigoplus_{i\in I}\psi_i$ commutes with the representation $\rho|_{K^1}\simeq\bigoplus_{j\in I} \nu^{(j)}$ which is multiplicity free. 
Therefore $\Psi(V^{j})= V^{j}$ for every $j\in I$, which implies that 
$\psi(V^{j})= V^{j}$ for $j\in I$. 
By Lemma \ref{conjcommutedisc}, this means that the anti-unitary operator $\psi$ commutes with $\rho(K^1)'_h$. 
\end{proof}

\begin{rem}\label{cond2'}
\rm{With Proposition \ref{multfreedual}, we obtain a special case of 
Theorem \ref{multfreegroup} where 
the representations on the fibers are discretely decomposable by changing condition (F) to:
\begin{itemize}
\item[\rm{(F'')}] For every $b\in B$ the restriction of $\rho$ to a subgroup $K^1_{(b)}\subseteq K_{(b)}$ is multiplicity free with irreducible decomposition $\rho|_{K^1_{(b)}}\simeq\bigoplus_{i\in I}\nu^{(i)}_b$ on $V=\bigoplus_{i\in I} V^{(i)}_b$ and $\bar{\nu}^{(i)}_b\circ\sigma\simeq\nu^{(i)}_b$ as representations of $K^1_{(b)}$.
\end{itemize}
}
\end{rem}

By Proposition \ref{multfreedual} and Remark \ref{samesubgroup} we obtain the 
following reformulation of Theorem~\ref{multfreegroup}. 


\begin{thm}\label{multfreegroup2}
Let $\V=G\times_{H,\rho} V\to M$ be a $G$-equivariant holomorphic vector bundle with a complex subalgebra $\q$ and an extension $\beta:\q\to \gl(V)$ of $\dd\rho$ defining a complex structure on it. Let $K\subseteq G$ be a subgroup 
and $\sigma$ be an automorphism of the Banach--Lie group $G$ stabilizing $K$ and $H$ such that $(\sigma,\psi)$ is a morphism between the representations $(\rho,\beta)$ and $(\bar{\rho},\bar{\beta})$. Suppose also that there is a subset $B$ of $G^{\sigma}$ such that:
\begin{itemize}
\item[\rm{(B)}]\label{cond41} The subset $KBH$ of $G$ satisfies $\overline{KBH}^\mathrm{o}\neq \emptyset$.
\item[\rm{(F)}]\label{cond42} For $N=Z_{K\cap H}(B)$ the restriction $\rho|_N$ is multiplicity free with irreducible decomposition $\rho|_N\simeq\bigoplus_{i\in I} \nu^{(i)}$ on $V=\bigoplus_{i\in I} V^{(i)}$ and $\bar{\nu}^{(i)}\circ\sigma\simeq\nu^{(i)}$ for $i\in I$ as representations of $N$.
\end{itemize} 
Then any unitary representation $(\pi,H)$ of $K$ 
is realized in $\Gamma (\V)$ is multiplicity free. 
\end{thm}

\section{Examples of weakly visible actions and propagation theorems}\label{exampleswv}

In this final section we discuss various concrete situations in the context 
of operator algebras and Hilbert spaces which fit into the 
setting developped above. 
In particular we exhibit several kinds of 
weakly visible actions on infinite dimensional spaces 
and  state some corresponding propagation theorems.

\subsection{Propagation theorem for linear base spaces}\label{proplinear}

Let $\H$ be a complex Hilbert space and let $\Fo(\H)\subseteq \O(\H)$ be the Fock space on $\H$ with reproducing kernel $K(x,y)=e^{\langle x,y\rangle}$ for $x,y\in \H$. Let 
$(e_j)_{j\in J}$ be an orthonormal basis of $\H$. 
Then the polynomial functions $p_m:\H\to\C$
\[ p_m(z)=\frac{1}{\sqrt{m!}}\prod_{j\in J}\langle z,e_j\rangle^{m_j}
\quad\mbox{for}\quad m\in \N_{0}^{(J)}\] 
form an orthonormal basis of $\Fo(\H)$, where $\N_{0}^{(J)}$ 
is the set of finitely supported tuples indexed by $J$ and $m!=\prod_{j\in J}m_j!$. 
The Hilbert subspace of $\Fo(\H)$ of homogeneous functions of degree $n\in\N_0$ 
has reproducing kernel $K_n(x,y)=\frac{1}{n!}\langle x,y\rangle$ for $x,y\in \H$.

Let $\delta:H\to\U(\H)$ be a norm-continuous representation of a Banach--Lie group $H$ and let $G:=\H\rtimes_{\delta}H$, so that $M:=G/H\simeq \H$. Let $\rho:H\to \U(V)$ be another norm-continuous representation, so we can define an associated 
(trivial) vector bundle
\[ q:\V=G\times_{\rho,H}V\to G/H\simeq \H.\] 
Its space of holomorphic sections $\Gamma(\V)$ can be identified with the space of holomorphic functions $\O(\H,V)$, where a section of the bundle 
$$s:(z,1)(\{0\}\times H)\mapsto [(z,1),f(z,1)]$$
is defined by a holomorphic function $f':\H\to V$, $f'(z)=f(z,1)$ for $z\in \H$. 

With the representation $\delta:H\to\U(\H)$ we define a representation $\delta':H\to\U(\Fo(\H))$, $(\delta'_hg)(z)=g(\delta_{h^{-1}}z)$ for $h\in H$, $g\in \Fo(\H)$ and $z\in \H$. We can realize the representation $\delta'\otimes\rho$ of $H$ in the space of holomorphic sections of the bundle via the embedding 
\begin{align*}
&\Fo(\H)\otimes V\to \O(\H,V)\simeq \Gamma(\V)\\
&f\otimes v\mapsto (z\mapsto f(z)v).
\end{align*}
For $z\in \H$ we have $\Ev_z:\Fo(\H)\otimes V\to V$, $f\otimes v\mapsto f(z)v$ and 
$(\Ev_z)^*:V\to \Fo(\H)\otimes V$, $(\Ev_z)^*(w)=e^{\langle \cdot,z\rangle}w$ so that the reproducing kernel is given by
\[ Q(y,z):=(\Ev_y)(\Ev_z)^*=(\Ev_y)e^{\langle \cdot,z\rangle}\id_V
=e^{\langle y,z\rangle}\id_V=e^{\langle y,z\rangle}\id_V.\] 


\begin{thm}\label{multfreegrouplinear}
Let $\sigma$ be an involution on $H$, and let $\sigma^{\H}$ be 
an anti-unitary operator on $\H$ such that
\[ \delta_{\sigma(h)}=\sigma^{\H}\circ\delta_{h} \circ (\sigma^{\H})^{-1}, \] 
and $\psi$ be an anti-unitary operator on $V$ such that
\[ \rho_{\sigma(h)}=\psi\circ\rho_{h} \circ \psi^{-1}\quad \mbox{ for } \quad h \in H.\]
Suppose further 
that there is a subset $S$ in the fixed point set of $\sigma^{\H}$ such that:
\begin{itemize}
\item[\rm{(B)}]\label{cond51} The closure of 
$\delta(H)S$ in $\cH$ has interior points. 
\item[\rm{(F')}]\label{cond52} The operator 
$\psi$ commutes with the hermitian operators in the commutant 
$\rho(Z_H(S))'$, where $Z_H(S) := \{ h \in H : (\forall s \in S)\ \delta_h(s)=s\}$ 
is the pointwise stabilizer of $S$ in~$H$. 
\end{itemize} 
Then the representation $\delta'\otimes\rho$ of $H$ on $\Fo(\H)\otimes V$ is multiplicity free. 
\end{thm}

\begin{proof}
For the subset $B=S\times\{1\}\subseteq\H\times\{1\}\subseteq G$, the relations 
$(b,1)(0,h)=(b,h)$ and $(0,h)(b,1)=(\delta_hb,h)$ for $h\in H$ and $b\in B$ lead 
to 
\[ Z_H(B)=\{h\in H:\delta_h(b)=b\quad\mbox{for}\quad b\in B\}.\] 

If $B$ is in the fixed point set of $\sigma^{\H}$, 
$HBH$ is dense in an open subset of $G$ 
(which is equivalent to $\oline{H.B}^\circ \not=\eset$ in $\cH$), 
and $\psi$ commutes with the hermitian operators in $\rho(Z_H(B))'$, 
then Theorem \ref{multfreegroup} with condition (F') in 
Remark~\ref{samesubgroup} implies that the representation $\delta'\otimes\rho$ of $H$ on $\Fo(\H)\otimes V$ is multiplicity free.
\end{proof}

\begin{cor}\label{cormultfreegrouplinear}
Let $\sigma$ be an involution on $H$, $\sigma^{\H}$ an anti-unitary operator on $\H$ such that
\[ \delta_{\sigma(h)}\circ \sigma^{\H}=\sigma^{\H}\circ\delta_{h}.\] 
Suppose further that $S \subeq \cH^{\sigma^\cH}$ is such that the closure of $\delta(H)S$ in $\cH$ 
has interior points. 
Then the representation $\delta'$ of $H$ on $\Fo(\H)$ is multiplicity free.
\end{cor}

\subsection{Hilbert-Schmidt operators}
In Section 5.6 of \cite{kob2} finite dimensional examples of visible actions on linear spaces are presented. We extend some of these results to the context of Hilbert--Schmidt operators. We denote by
 $B_2(\H_1,\H_2)$ the Hilbert--Schmidt operators from $\H_1$ to $\H_2$ and by  
\[ \U_2(\H):=\U(\H)\cap(\bone+B_2(\H)) \]
the unitary Hilbert--Schmidt perturbations of the identity. The first result is about torus actions and the remaining results involve an approximate Cartan decomposition in this context. 

The two sided action of the group $\U_2(\H)\times \U_2(\H)$ on $\GL_2(\H)$ or $B_2(\H)$ is weakly visible. We take an orthonormal basis $(e_n)_{n\in \N}$ of $\H$ and define a conjugation on $\H$ by $J(\sum_n\alpha_ne_n)=\sum_n\overline{\alpha_n}e_n$ and an automorphism of $\GL_2(\H)$ by $\sigma(g)=J gJ$. If we define  $S$ as the subset of positive diagonal operators in $\GL_2(\H)$, the Cartan decomposition and a 
finite-dimensional approximation argument 
imply that  $D:=(\U_2(\H)\times \U_2(\H)).S=\U_2(\H)S \U_2(\H)$ is dense in $\GL_2(\H)$ and also in $B_2(\H)$. Furthermore the $\U_2(\H)\times \U_2(\H)$-orbits are $\sigma$-invariant 
since $\sigma(\U_2(\H))=\U_2(\H)$, and $\sigma(s)=s$ for $s\in S$. 
Therefore the action is $(S,\sigma)$-weakly visible.

\begin{ex}\label{torushilbert}
\rm{
A particular case is the multiplication action of the abelian Banach--Lie group 
\[ \ell^2(\N,\R)/\Z^{(\N)} \cong 
\exp(i\ell^2(\N,\R))\subseteq \U_2(\ell^2(\N,\C)) 
= \U(\ell^2(\N,\C)) \cap (\bone + B_2(\ell^2(\N,\C)) \] 
on $\ell^2(\N,\C)$. It is $S$-weakly-weakly visible if we take $S=\ell^2(\N,\R)$ and $\sigma$ is conjugation on $\ell^2(\N,\C)$. Here the finitely supported sequences $f:\N\to \C$ are contained in $D:=\exp(i\ell^2(\N,\R)).\ell^2(\N,\R)$, so that this 
subset is dense in $\ell^2(\N,\C)$. 
The sequence $f(n)=\frac{i}{2^n}$ is contained in $\ell^2(\N,\C)$ 
but not in $D$ because $(-i,-i,-i,\ldots)$ is not contained in $\exp(i\ell^2(\N,\R))$. 

Assume $H:=\exp(i\ell^2(\N,\R))$ and $\ell^2(\N,\C)$ are endowed with the canonical involutions. Corollary~\ref{cormultfreegrouplinear} implies that the representation of $H$ on $\Fo(\ell^2(\N,\C))\subeq \cO(\ell^2(\N,\C))$ is multiplicity free. This is the fact that in the Taylor expansion of a function $f\in\Fo(\ell^2(\N,\C))$ each monomial appears exactly once. 
In the finite-dimensional context this is 
the most basic example of a multiplicity-free representation 
(see the introduction of \cite{kob2}). }
\end{ex}

\begin{rem}
\rm{
Let $\H=L^2(X,\mu)$ with a $\sigma$-finite measure space $(X,\mu)$. 
Then the multiplication algebra $\A=L^{\infty}(X,\mu)$ is maximal abelian 
in $B(L^2(X,\mu))$. If $\sigma(f) = \overline{f}$ is 
complex conjugation on $L^2(X,\mu)$ and $S=L^2(X,\mu;\R)$ denotes the subspace of 
real-valued functions, then the action of the unitary group 
$\U_{\A}$ of $\cA$ on $L^2(X,\mu)$ is $(S,\sigma)$-visible. 
If $\mu$ is infinite, then the action of the subgroup 
$\exp(i L^2(X,\mu;\R))\subseteq \U_{\A}$ is $(S,\sigma)$-weakly visible.
}
\end{rem}

\begin{ex}\label{linearmfhs}
\rm{Let $\H_1$ be a closed subspace of the Hilbert space $\H_2$ and take an orthonormal basis $(e_i)_{i\in I_1}$ of $\H_1$ and an orthonormal basis $(e_i)_{i\in I_2}$ of $\H_2$ such that $I_1\subseteq I_2$. Consider the action $\delta$ of $H:=\U_2(\H_1)\times \U_2(\H_2)$ on $M=B_2(\H_1,\H_2)$ given by 
  \begin{equation}
    \label{eq:doubleact}
\delta(u_1,u_2)(x)=u_2xu_1^{-1}.
  \end{equation}

Consider the subset 
\[ S:=\{A\in  B_2(\H_1,\H_2):(\forall j\in I)\ A(e_j)\in\R e_j\} \] 
and the conjugations of $\H_1$ and $\H_2$ given by 
$J_2(\sum_{j\in I_2}\alpha_j e_j)=\sum_{j\in I_2}\overline{\alpha_j}e'_j$ 
and $J_1=J_2|_{\H_1}$ respectively. Define a conjugation on $B_2(H_1,\H_2)$ 
by $\sigma^M(x)=J_2 xJ_1$. It is easy to verify that the action 
\eqref{eq:doubleact} 
is $(S,\sigma^M)$-weakly visible and compatible with the automorphism of $\U_2(\H_1)\times \U_2(\H_2)$ given by $\sigma(u_1,u_2)=(J_1u_1J_1,J_2u_2J_2)$. 
Corollary~\ref{cormultfreegrouplinear} implies that the representation of $H$ on $\Fo(B_2(\H_1,\H_2))\subeq \cO(B_2(\H_1,\H_2))$ is multiplicity free.
}
\end{ex}

\begin{ex}\label{kacexample}
\rm{Let $\H_1$ be a closed subspace of the Hilbert space $\H_2$ and take an orthonormal basis $(e_i)_{i\in I_1}$ of $\H_1$ and an orthonormal basis $(e_i)_{i\in I_2}$ of $\H_2$ such that $I_1\subseteq I_2$. 
Consider the action $\delta$ of $\U_2(\H_1)\times \U_2(\H_2)$ on 
$M := B_2(\H_1,\H_2\oplus\C)\simeq B_2(\H_1,\H_2)\oplus~\H_1$ given by 
\[ \delta(u_1,u_2)(x,\xi)=(u_2xu_1^{-1},u_1\xi),\]
the subset 
\[S':=S\oplus \H_1^{J_1}, \quad \mbox{ where } \quad 
\H_1^{J_1} :=\{\xi\in\H_1:J_1\xi=\xi\} 
\] 
with $S$ as in Example \ref{linearmfhs}, 
and the conjugation on $B_2(\H_1,\H_2)\oplus\H_1$ 
given by $\sigma(x,\xi)=(J_2xJ_1,J_1\xi)$.
The action is $(S',\sigma)$-weakly visible since the action of the subgroup
\begin{align*}
\{(u,(u,\id_{\H_1^{\perp}}))\in\U_2(\H_1)\times \U_2(\H_2):& 
(\forall j \in I_1)\ u(e_j)\in \T e_j\}
\end{align*}
fixes $S$ and rotates the vectors in a dense subset of $\H_1$ into $\H_1^{J_1}$.

Corollary~\ref{cormultfreegrouplinear} implies that the action of $H$ on 
$$\Fo(B_2(\H_1,\H_2)\oplus\H_1)\simeq \Fo(B_2(\H_1,\H_2))\otimes\Fo(\H_1) \subeq \cO(B_2(\H_1,\H_2)\oplus\H_1)$$ 
is multiplicity free.  
}
\end{ex}

\begin{rem}
\rm{
Example~\ref{linearmfhs}
 can be interpreted as the isotropy representation of the group $\U_2(\H_1)\times \U_2(\H_2)$ on the tangent space at the base point of the restricted Gra\ss{}mannian 
\[ G_{\rm res}(\cH_1\oplus \cH_2) \cong \U_2(\H)/(\U_2(\H_1)\times \U_2(\H_2)).\] 
}
\end{rem}

\begin{ex}\label{linearmfwedge}
\rm{
Assume the context and notation of Example \ref{linearmfhs}, 
and for simplicity we set $\H_1 = \cH_2 = \cH$. The action $\delta$ of 
$H:=\U_2(\H)\times \U_2(\H)$ on 
$\cK := B_2(\cH)$ given by 
\[  \delta(u_1,u_2)(x)=u_2xu_1^{-1} \] 
is a  unitary representation $\delta:H\to \U(\cK)$. 
As uniformly bounded isotropy representations, 
we can take for example the representations 
\[ \rho(u_1, u_2) := 
\rho_1(u_1) \oplus \rho_2(u_2),\quad \mbox{ where } \quad
\rho_j:= \bigoplus_{k=1}^{m_j}\Lambda^{n_{j,k}}.\] 
Here, for $n\in\N$ the representation $\Lambda^{n}:\U_2(\H)\to \U(\Lambda^{n}\H)$ 
is defined by
\[ \Lambda^{n}(u)(v_1\wedge\ldots\wedge v_n)=uv_1\wedge\ldots\wedge uv_n. \] 
Consider the group 
\[ 
N=Z_H(B)=\{(u,u)\in \U_2(\H_1)\times 
\U_2(\H_2):(\forall j \in I)\ u(e_j)\in \T e_j\},\] 
where
\[ B=\{x\in  B_2(\H_1,\H_2):x(e_i)\in\R e'_{i}\quad\mbox{for}\quad j\in J\}.\] 
Since 
\[ \Lambda^{n}(\diag (t_j)_{j\in J})(e_{j_1}\wedge\ldots\wedge e_{j_n})
=t_{j_1}\ldots t_{j_n}(e_{j_1}\wedge\ldots\wedge e_{j_n}),\] 
if the $n_{j,k}$ for $j=1,2$ and $k=1,\ldots,m_j$ are all distinct, 
$\rho|_N$ is multiplicity free. 
We can take the canonical anti-unitary operators given by complex conjugation on 
\[ \Lambda^{n}\big(\ell^2(J)\big)\subseteq \bigotimes^{n}\ell^2(J)=\ell^2(J^n).\] 
Theorem \ref{multfreegrouplinear} implies that the representation $\delta'\otimes\rho$ of $H=\U_2(\H)\times \U_2(\H)$ is multiplicity-free. 
}
\end{ex}

\begin{ex}\label{linearmftrace}
\rm{For a Hilbert space $\cH$, let 
$\U_1(\H) :=\U(\H)\cap(1+B_1(\H))$ 
denote the group of unitary operators which are trace class perturbations of 
the identity. 
In the previous example we can take instead of 
$\U_2(\H)\times \U_2(\H)$ the group $\U_1(\H)\times \U_1(\H)$   
 to construct isotropy representations from the operator 
determinant. The isotropy representation is 
$\rho(u_1, u_2) = \rho_1(u_1) \oplus \rho_2(u_2)$, where   
$$\rho_j = \bigoplus_{k=1}^{m_j}(\deter_{\U_1(\H)})^{n_{j,k}}.$$
If the $n_{j,k}$ for $j=1,2$ and $k=1,\ldots,m_j$ are all distinct, then $\rho|_N$ 
is multiplicity free, where
\[ N=Z_H(B)=\{(u,u)\in \U_1(\H_1)\times 
\U_1(\H_2): (\forall j \in I)\ u(e_j)\in \T e_j\}.\] 
}
\end{ex}

\subsection{Finite von Neumann algebras} 

Some basic examples of type II$_1$ factors have representations such that the Hilbert space where 
the factor $\cM$ acts can be endowed with an anti-unitary operator $J$ such  
that 
\begin{equation}
  \label{eq:j-cond}
J\cM J = \cM \quad \mbox{ and } \quad Jx J = x\quad\mbox{ for }x\in\cA_h 
\end{equation}
for a maximal abelian $*$-subalgebra $\cA$ of $\cM$ (a {\it masa} for short), see 
Examples~\ref{exvncartype} and~\ref{exgroupmeasurespace}. We are going to use Theorem \ref{multfreegrouplinear} to construct examples of representations of the product 
$H := \U_{\cM}\times\U_{\cM}$ of the unitary group of a finite factor $\cM$, where the base is the GNS construction of $\cM$, the slice $S$ 
consists of the hermitian operators in a masa, and 
the conjugations are constructed from the conjugation $J$. 

We denote by $L^2(\M)$ the GNS Hilbert space, which is the completion of $\M$, 
endowed with the inner product $\langle x,y\rangle=\tau(xy^*)$ for $x,y\in \cM$, where $\tau$ is the trace of the algebra~$\M$. Let the unitary representation $\delta:H\to \U(\H)$ of $H=U_{\M}\times U_{\M}$ on $\H=L^2(\M)$ given by  
\[ \delta(u_1,u_2)(x)=u_1 xu_2^{-1} 
\quad \mbox{ for } \quad u_1,u_2\in U_{\M}, x\in \M\subseteq L^2(\M).\]
Let $J$ be as in \eqref{eq:j-cond}. 
If we define $\sigma(x)=J xJ$ for $x\in \M$, then $\sigma$ extends to an anti-unitary involution $\sigma^{L^2(\M)}$ on $L^2(\M)$ compatible with the involutions on $U_{\M}$ and $H=U_{\M}\times U_{\M}$ given by $\sigma(u)=J uJ$ for $u\in U_{\M}$, and $\sigma(h)=(J,J) h(J,J)$ for $h\in H=U_{\M}\times U_{\M}$, respectively. 
In fact, for $u_1,u_2\in U_{\M}$ and $x\in \M\subseteq L^2(\M)$, we have 
\begin{align*}
(\sigma(u_1),\sigma(u_2)).\sigma^{L^2(\M)}(x)&=\sigma(u_1)\sigma^{L^2(\M)}(x)\sigma(u_2^{-1})=(J u_1J)J xJ(J u_2^{-1}J)\\
&=J u_1xu_2^{-1}J=\sigma^{L^2(\M)}(u_1xu_2^{-1})
=\sigma^{L^2(\M)}((u_1,u_2).x).
\end{align*}

Note that $\sigma^{L^2(\M)}$ fixes $S:=\A_h$ pointwise. 
To prove that the action of $\U_{\M}\times \U_{\M}$ on $L^2(\M)$ is 
$(S,\sigma^{L^2(\M)})$-weakly visible it remains to show that $D:=U_{\M}\A_h U_{\M}$ is dense in $L^2(\M)$. This is a consequence of the following proposition. 

\begin{prop}
Let $\M$ be type II$_1$-factor with trace $\tau$, 
let $\A\subseteq\M$ be a masa, 
and denote the cone of positive invertible elements in $\A$ by $\A^+$. Then $D:=U_{\M}\A^+ U_{\M}$ is dense in $L^2(\M)$
\end{prop}

\begin{proof}
For a self-adjoint operator $x\in \M$ there exists a unique Borel probability 
measure $m_x$ on $\R$ such that 
$$\int_{-\infty}^{+\infty}\lambda^ndm_x(\lambda)=\tau(x^n)\quad\mbox{for}\quad n\in \{0\}\cup\N.$$
If, conversely, $m$ is a compactly supported probability measure on $\R$, 
then there is a self-adjoint operator $x\in \A$ with spectral measure $m_x = m$, 
see \cite{ak06}*{Prop. 5.2}. The norm closure of the unitary orbit $\O(x)=\{uxu^{-1}:u\in \U_{\M}\}$ of a self-adjoint $x\in \M$ consists of the selfadjoint operators in $\M$ 
with the same spectral measure as $x$ (\cite{ak06}*{Thm.~5.4}), 
so the spectral measure is a complete invariant of approximate unitary equivalence. 

Then the polar decomposition implies that 
$D:=U_{\M}\A^+ U_{\M}$ is dense in the unit group $\M^\times$ of~$\M$. 
In \cite{ch70}*{Thm. 5} it is shown that a von Neumann algebra $\cM$ 
has dense unit group in the norm topology if and only if $\M$ is of finite type.
Therefore, if $\M$ is a finite type factor, then $D$ 
is norm dense in $\M$.  Since on $\M\subseteq L^2(\M)$, 
the Hilbert space norm is dominated by the uniform norm, 
$D:=U_{\M}\A^+ U_{\M}$ is dense in $L^2(\M)$ with its Hilbert space topology. 
\end{proof}

\begin{rem}{\rm 
As the example after \cite{ak06}*{Thm.~5.4} shows, 
there exist selfadjoint operators whith the same spectral measure 
which are not unitarily equivalent. }
\end{rem}

Set $S:=\A_h$, so that 
\[ 
N=Z_H(S\times\{1\})=\{(t,t)\in U_{\M}\times U_{\M}:t\in U_{\A}\}.\]
If there is a unitary representation $\rho:U_{\M}\times U_{\M}\to U(V)$ on a Hilbert space $V$, and there is an antiunitary involution $\psi$ on $V$ such that 
\begin{itemize} 
\item $\rho_{\sigma(h)}\circ \psi=\psi\circ\rho_{h} \quad \mbox{ for } \quad h \in H$.
\item $\psi$ commutes with the hermitian operators in 
$\rho(\{(t,t)\in U_{\M}\times U_{\M}:t\in U_{\A}\})'$ 
\end{itemize} 
then Theorem \ref{multfreegrouplinear} implies that the representation $\delta'\otimes\rho$ of $H=U_{\M}\times U_{\M}$ on $\Fo(L^2(\M))\otimes V$ is multiplicity free.  

We now turn to examples of factors $\cM \subeq B(\cH)$ 
for which there exists an anti-unitary operator $J$ such  
that $J\cM J = \cM$ and $Jx J = x$ for selfadjoints $x$ in some masa $\cA \subeq \cM$.

\begin{defn}{\rm 
There are two types of masas $\cA$ in a II$_1$ factor $\cM$ specified
 in terms of the algebra generated by the normalizer of the masa. The normalizer of the masa is defined by
$$N(\cA)=\{u\in U_{\cM}:u\cA u^*=\cA\}.$$
The masa $\cA$ is called {\it regular} or {\it Cartan} 
if $N(\cA)''=\cM$, and it is called singular if $N(\cA)''=\cA$, see the first paragraph of \cite{ss08}*{3.2}. }
\end{defn}

\begin{ex}\label{exvncartype}
\rm{
(von Neumann algebras of CAR type). Let $M_2(\C)$ be the $2\times 2$ complex matrices with diagonal masa $D$. For each $k\in \N$ let $M_k$ be a copy of 
$M_2(\C)$ with a copy of $D$ as diagonal masa $D_k$. The algebra 
$\cM_n=\otimes_{k=1}^{n}M_k\simeq M_{2^n}(\C)$ has a masa 
$\cA_n=\otimes_{k=1}^{n}D_k$. We have embeddings $\cM_n\hookrightarrow\cM_{n+1}$, $x\mapsto x\otimes \id_{M_2(\C)}$ which preserve the normalized traces. Therefore 
$\bigcup_{n=1}^{\infty}\otimes_{k=1}^nM_k$ is a $*$-algebra with trace
\[ \tau(\otimes_{k=1}^{\infty}x_k)=\prod_{k=1}^{\infty}\tr(x_k),\]  
where all but finitely many $x_i$ are equal to the identity of $M_2(\C)$. The weak closure $\cM$ of the GNS representation of $\bigcup_{n=1}^{\infty}\otimes_{k=1}^nM_k$ is a copy of the hyperfinite II$_1$-factor and the weak closure $\cA$ of $\bigcup_{n=1}^{\infty}\otimes_{k=1}^nD_k$ 
in this representation is a Cartan masa, see the first part of Subsection 3.4 in \cite{ss08}. The canonical complex conjugation on $\bigcup_{n=1}^{\infty}\otimes_{k=1}^nM_k$ 
yields a complex conjugation $J$ on the GNS space such that $J\M J=\M$, $Jx J=x$ for $x\in\cA_h$ and $\tau(JxJ)=\overline{\tau(x)}$. 
}
\end{ex}

\begin{rem} 
\rm{
In \cite{st80} it is shown that there is, up to conjugacy, a unique real 
von Neumann algebra $\cR$ in the hyperfinite factor $\cM$ of type II$_1$, i.e. $\cR$ is a $*$-algebra over the reals such that $\cR+\sqrt{-1}\cR=\cM$ and $\cR\cap\sqrt{-1}\cR=\{0\}$. It follows that any two involutive antilinear automorphisms of this factor are conjugate under $\Aut(\cR)$. This contrasts with the situation in $B(\H)$, where there are two distinct conjugacy classes of such automorphisms, induced by conjugation with antiunitary operators $J$ satisfying 
$J^2 = \pm \bone$. 
}
\end{rem}

\begin{ex}\label{exgroupmeasurespace}
\rm{
(Group-measure space construction). See Subsection 8.6 in \cite{kr97} for detailed information on this construction. Let $\Gamma\curvearrowright (X,\mu)$ be a probability measure preserving action of a countable discrete group $\Gamma$. We 
consider the unitary {\it Koopman representation}
\[ \alpha:\Gamma\to\U(L^2(X,\mu)), \qquad (\alpha_sf)(x)=f(s^{-1}.x)
\quad \mbox{  for } \quad s\in\Gamma, f\in L^2(X,\mu), x\in X.\]
Let $\lambda:\Gamma\to\U(\ell^2(\Gamma))$ be the left regular representation. We regard $L^{\infty}(X)\simeq L^{\infty}(X)\otimes 1\subseteq B(L^2(X)\otimes \ell^2(\Gamma))$. For 
$s\in \Gamma$ we consider the unitaries $u_s=\alpha_s\otimes \lambda_s\in\U(L^2(X)\otimes \ell^2(\Gamma))$. The crossed product von Neumann algebra is defined as
\[ L^{\infty}(X)\rtimes\Gamma :=\Big\{\sum_{s\in \Gamma}a_su_s:
a_s\in L^{\infty}(X)\Big\}''\subseteq B(L^2(X)\otimes \ell^2(\Gamma)).\]
The trace is given by the extension of
\[ \tau\bigg(\sum_{s\in\Gamma}a_su_s\bigg)=\int_Xa_ed\mu.\]
The action is called {\it free} if $\A:=L^{\infty}(X)$ is maximal abelian in $L^{\infty}(X)\rtimes\Gamma$, and in this case $L^{\infty}(X)$ is a Cartan subalgebra. 
The von Neumann algebra 
$L^{\infty}(X)\rtimes\Gamma$ is a II$_1$ factor if and only if the action $\Gamma\curvearrowright X$ is ergodic, i.e., $L^{\infty}(X,\mu)^{\Gamma}=\C 1$.

Let $J$ be the complex conjugation on $L^2(X)\otimes \ell^2(\Gamma)\simeq L^2(X\times \Gamma)$. Observe that $J a J=\overline{a}$ for $a\in L^{\infty}(X)$ and $J u_s J=u_s$ for $s\in \Gamma$. Therefore $J(L^{\infty}(X)\rtimes\Gamma)J=L^{\infty}(X)\rtimes\Gamma$ and $J x J=x$ for $x\in\cA_h$. Note that $\tau(J x J)=\overline{\tau(x)}$ for $x\in L^{\infty}(X)\rtimes\Gamma$.  
}
\end{ex}

\subsection{Symmetric spaces}

In Subsections 5.3 and 5.4 of \cite{kob2} finite dimensional examples of visible actions on symmetric spaces are presented. The action of $H$ on $G/H$ 
for a vast class of symmetric spaces is 
studied in \cite{kob08}. We begin by presenting weakly visible actions on Gra\ss{}mannians and symmetric domains modeled on Banach spaces. We first present a weak visibility result for the group of bounded invertible operators acting on a Hilbert space to illustrate the approximate Cartan decomposition which is 
involved in the subsequent arguments.

\begin{ex}\label{cartanunif}
\rm{
Let $\H$ be a Hilbert space with orthonormal basis $(e_j)_{j\in I}$ and define a complex conjugation on $\H$ by $J(\sum_j\alpha_j e_j )=\sum_j \overline{\alpha_j}e_j$ 
and an automorphism on $\GL(\H)$ by $\sigma(g)=J gJ$. If we define  $S$ as the set of 
positive diagonal operators in $\GL(\H)$ then $D:=(\U(\H)\times \U(\H)).S=\U(\H)S \U(\H)$ is dense in $\GL(\H)$: 
For $g\in\GL(\H)$ we have the polar decomposition $g=up$, with $u\in \U(\H)$ and an 
invertible $p>0$. Given $\epsilon>0$ we can use the measurable 
functional calculus of $p$ to find a positive invertible $q$ with finite spectrum such that $\|q-p\|<\epsilon$. Since $q$ has finite spectrum, there is a unitary $v$ such that $s:=vqv^{-1}\in S$. Hence 
$$\|g-(uv^{-1})sv\|=\|up-uq\|<\epsilon.$$ 
Furthermore the $\U(\H)\times \U(\H)$-orbits are preserved under $\sigma$ since $\sigma(\U(\H))=\U(\H)$, and $\sigma(s)=s$ for $s\in S$, so the action 
is $(S,\sigma)$-weakly visible.
}
\end{ex}

\begin{ex}\label{pcss}
\rm{Let $\cH_1$ and $\cH_2$ be complex Hilbert spaces. 
We consider the identical representation of $\U(\K)$ on the complex Hilbert space 
$\K=\H_1\oplus\H_2$. Then the subgroup $Q:=\{g\in\GL(\K):g\H_1=\H_1\}$ is a complex Lie subgroup of $\GL(\K)$ and the Gra\ss{}mannian $\Gr_{\H_1}(\K):=\GL(\K)\H_1\simeq\GL(\K)/Q$ 
carries the structure of a complex homogeneous space on which the unitary group 
$G=\U(\K)$ acts transitively\begin{footnote}{Here we use that, 
for $\cE := g\cH_1$, the orthogonal space $\cE^\bot$ is the image of 
$\cH_2$ under $(g^{-1})^*$. Hence there exists a unitary isomorphism 
$\cE \oplus \cE^\bot \to \cH_1 \oplus \cH_2$ mapping $\cE$ to $\cH_1$.}
\end{footnote}
and which is isomorphic to $G/H$ for $H:=\U(\K)_{\H_1}\simeq\U(\H_1)\times\U(\H_2)$. 
Writing elements of $B(\H)$ as $(2\times 2)$-matrices according 
to the decomposition $\K=\H_1\oplus \H_2$, we have
\[ \q=
\Big\{ \pmat{a & b \\ c & d} \: a\in B(\H_1),b\in B(\H_2,\H_1),d\in B(\H_2)\Big\}.\] 
and $\q\simeq\p^+\rtimes\h_{\C}$ and $\gl(\K)=\q\oplus\p^-$ holds for
\[ \p^-=
\Big\{ \pmat{0 & 0 \\ c & 0} \:c\in B(\H_1,\H_2) \Big\}\quad \mbox{ and }\quad 
\p^+=\overline{\p^-}.\] 

Assume that $\H_1 := \H \oplus \{0\}$ and 
$\H_2 := \{0\} \oplus \H$ for a Hilbert space $\H$. We fix an orthonormal basis 
$(e_i)_{i \in I}$ in $\H$. We can define a conjugation $J$ on $\H$ 
by $J(\sum_{j\in I}\alpha_j e_j)=\sum_{j\in I}\overline{\alpha_j}e_j$ 
and defined $J_{\K}(v,w) := (Jv,Jw)$ on $\K$. 
Let $\sigma(u)=J_{\K} uJ_{\K}$ be the corresponding 
involution on $\U(\K)$ and $\sigma^M(u\H_1)=J u\H_1=\sigma(u)\H_1$ be 
the corresponding involution on the Gra\ss{}mannian $\Gr_{\H_1}(\K)$.

Consider the embedding 
\[ \tau:\B(\H)\to \g,\quad A\mapsto 
\pmat{0 & A \\ -A^* & 0},\] 
the diagonal real subalgebra of $\B(\H)$ given by
\[ \cA_{\R}=\{A\in\B(\H): (\forall j \in I)\ 
A(e_j)\in\R e_j\},\]
and let $B:=\exp(\tau(A_{\R}))$, where $\exp:\g\to G$ is the exponential 
map of $\U(\K)$. Observe that $B\subseteq G^{\sigma}$ and that   
\begin{align*}
Z_{H}(B)=\{(u,u)\in\U(\H)^2 \: (\forall j \in I)\ 
u(e_j) \in \T e_j\}.
\end{align*}
To prove that the action of $K=H=\U(\H)^2$ 
on $\Gr_{\H_1}(\K)$ is $(B,\sigma^M)$-weakly-visible it remains to show that 
$$\overline{\U(\H)^2\exp(\tau(\cA_{\R}))\U(\H)^2}^\mathrm{o}\neq \emptyset.$$ 
Observe that 
\[ HBH=\U(\H)^2\exp(\tau(\U(\H)\cA_{\R}\U(\H)))\U(\H)^2.\] 
For $\eps >0$, the argument used to prove the approximate Cartan decomposition of 
Example~\ref{cartanunif} leads to 
\[ \overline{\U(\H)(\cA_{\R}\cap B_{\eps/2}(\eps \bone))\U(\H)}
=B_{\leq\eps/2}(\eps\bone).\]
We choose $\eps >0$ small enough to ensure that 
$\tau(B_{\eps/2}(\eps\bone))$ is contained in a neighborhood of 
$0$ on which $\exp$ is a local diffeomorphism, and the result follows. }
\end{ex}

\begin{ex}\label{ncss}
\rm{
Let $\K=\H\oplus\H$,  $\H_1 := \H \oplus \{0\}$, and 
$\H_2 := \{0\} \oplus \H$ for a Hilbert space~$\H$. We fix an orthonormal basis 
$(e_j)_{j \in I}$ in $\H$. We can define a conjugation $J$ on $\H$ 
by $J(\sum_{j\in I}\alpha_j e_j)=\sum_{j\in I}\overline{\alpha_j}e_j$ 
and define $J_{\K}(v,w) := (Jv,Jw)$ on $\K$.

We endow the Hilbert space $\K$ 
with the indefinite hermitian form given by 
\[ h((v_1,v_2),(w_1,w_2))=\langle v_1,w_1 \rangle
-\langle v_2, w_2\rangle .\]
 We can write $\D:=\{z\in B(\cH):\|z\|<1\}$ as $G/H$, where $G$ is the pseudo-unitary group 
\[ G=\U(\H_1,\H_2)=\{g\in \GL(\K):h(g.v,g.v)=h(v,v)\quad\mbox{for all}\quad 
v\in\K\},\] 
and $H=\U(\H_1)\times\U(\H_2)\cong \U(\cH)^2$ is the subgroup of diagonal matrices in $G$. In fact, $G$ acts transitively on $M$ by fractional linear transformations $\bigl(\begin{smallmatrix}
a&b \\ c&d
\end{smallmatrix} \bigr).z=(az+b)(cz+d)^{-1}$, where the $(2\times 2)$-block matrix is written according to the decomposition of $\K$.  The stabilizer of $0\in \D$ 
is the group $H$. 

A conjugation on $Z:=B_2(\H,\H)$ is given by $\sigma^Z(x)=J xJ$. 
Let 
\[ \tau:\B(\H)\to \g,\quad A\mapsto  
\pmat{0 & A \\ A^* & 0}.\] 
Consider the diagonal real subalgebra of $\B(\H)$ given by
\[ \cA_{\R}=\{A\in\B(\H): (\forall j \in I)\ A(e_j)\in \R e_j\}, \] 
and let $B:=\exp(\tau(\cA_{\R}))$, where $\exp(X) = \sum_{n = 0}^\infty \frac{X^n}{n!}$ 
is the operator exponential. 
Observe that $B\subseteq G^{\sigma}$ and that   
\begin{align*}
Z_{H}(B)=\{(u_1,u_2)\in\U(\H)\times \U(\H):& 
(\forall j \in I)\ u_1(e_j)\in \T e_j, u_2(e_j)\in \T e_j\}. 
\end{align*}
To prove that the action of $K=H=\U(\H_1)\times \U(\H_2)$ on $D$ is $(B,\sigma^M)$-weakly-visible it remains to show that 
$$\overline{\U(\H)^2.\exp(\tau(\cA_{\R})).\U(\H)^2}^\mathrm{o}\neq \emptyset.$$
This follows from an argument as in the last part of Example \ref{ncss}. The action is also compatible with the automorphism of $\U(\H_1,\H_2)$ given by $\sigma(u)=J_{\K}uJ_{\K}$.
}
\end{ex}

\begin{ex}\label{pcssr}
\rm{
Let $\K$ be a complex Hilbert space which is a direct sum $\K=\H_1\oplus\H_2$.
\[ B_{\res}(\K):=
\Big\{\pmat{a & b \\ c & d} \in B(\K):c\in B_2(\H_1,\H_2),b\in B_2(\H_2,\H_1) \Big\}\] 
is a complex Banach-$*$-algebra. Its unit group is 
$$\GL_{\res}(\K)=\GL(\K)\cap B_{\res}(\K).$$
The restricted unitary group is 
$$\U_{\res}(\K)=\U(\K)\cap \GL_{\res}(\K).$$
The homogeneous space $\Gr_{\res}:=\U_{\res}(\K)/(\U(\H_1)\times\U(\H_2))$ is the restricted Gra\ss{}mannian. There is an isomorphism
$$\U_2(\K)/(\U_2(\H_1)\times\U_2(\H_2))\simeq \U_{\res}(\K)_0/(\U(\H_1)\times\U(\H_2)),$$
where $\U_{\res}(\K)_0$ is the connected component of $\U_{\res}(\K)$ given by operator 
$(2\times 2)$-block matrices in $\U_{\res}(\K)$ with diagonal operators with Fredholm index equal to zero.

As in the previous example we assume that $\H_1 := \H \oplus \{0\}$ and 
$\H_2 := \{0\} \oplus \H$ for a Hilbert space $\H$. Let 
\[\tau: B_2(\H)\to \g,\quad A\mapsto  
\pmat{0 & A \\ -A^* & 0}.\] 
Consider the diagonal real subalgebra of $B_2(\H_1)$ given by
\[ \cA_{\R}=\{f\in B_2(\H): f(e_i)\in\R e_i\quad\mbox{for}\quad i\in I\},\] 
and let $B:=\exp(\tau(\cA_{\R}))$, where $\exp:\g\to G$ is the exponential map of $\U_2(\H)$. Observe that $B\subseteq G^{\sigma}$ and that   
\begin{align*}
Z_{H}(B)=\{(u,u)\in\U_2(\H)^2:& u(e_i)\in \T 
 e_i \quad \mbox{for}\quad i\in I\}.
\end{align*}
To prove that the action of $K=H=\U_2(\H_1)\times \U_2(\H_2)$ on the restricted Gassmannian is $(B,\sigma^M)$-weakly-visible it remains to show that 
$$\overline{\U_2(\H)^2.\exp(\tau(\cA_{\R})).\U_2(\H)^2}^\mathrm{o}\neq \emptyset.$$ 
Observe that 
\[ HBH=\U_2(\H)^2\exp(\tau(\U_2(\H)\cA_{\R}\U_2(\H)^2))
\U_2(\H)^2.\] 
A finite dimensional approximation argument leads to 
\[ \overline{\U_2(\H)\cA_{\R}\U_2(\H)}=B_2(\H),\] 
and the result follows.
}
\end{ex}

\begin{ex}\label{ncssr}
\rm{
The restricted pseudo-unitary group is 
$$G=\U_{\res}(\H_1,\H_2)=\U(\H_1,\H_2)\cap\GL_{\res}(\H).$$
It acts transitively on 
$$\D=\{z\in B_2(\H_2,\H_1):\|z\|<1\}$$
by fractional linear transformations. The stabilizer of $0$ in $G$ is $H=\U(\H_1)\times\U(\H_2)$.

We can adapt the previous example to this restricted group context as we did with both examples of positively curved symmetric spaces.
}
\end{ex}

\subsection{Approximate triunity}

In \cite{kob04} three multiplicity-free results are shown to stem from a single geometry. The basic result is the following, which we state in a form suited to our infinite dimensional context. 

\begin{lem}\label{approxtriun}
If $G$ is a topological group with subgroups $K$, $B$ and $H$, then the following conditions are equivalent:
\begin{enumerate}
\item[\rm(1)] $KBH$ is dense in $G$.
\item[\rm(2)]  $HBK$ is dense in $G$.
\item[\rm(3)]  $\diag(G)(B\times B)(K\times H)$ is dense in $G\times G$.
\end{enumerate}
\end{lem}

\begin{proof}
The equivalence of (1) and (2) is trivial. If (3) holds, 
then the image $KBH$ of $\diag(G)(B\times B)(K\times H)$ 
under the continuous map $\tau \: G \times G \to G, 
\tau(g,h) := g^{-1}h$ is dense in $G$, which is (1).

Now we assume that (1) and (2) hold. From 
$(hbk,e)=(h,h)(b,e)(k,h^{-1})$ 
we conclude that 
\[ \overline{HBK}\times\{e\}\subseteq \overline{\diag(G)(B\times B)(K\times H)}.\] 
From a similar equation we conclude that
\[ \{e\} \times\overline{KBH}\subseteq\overline{\diag(G)(B\times B)(K\times H)},\] 
so that (1) and (2) imply (3).
\end{proof}

The next theorem was proved in \cite{kob04}*{Thm. 3.1}: 

\begin{thm} \label{thm:6.24}
Let $n_1+n_2+n_3=p+q=n$, and consider the naturally embedded groups $K:=\U(n_1)\times\U(n_2)\times\U(n_3)$ and $H:=\U(p)\times\U(q)$ in $G:=\U(n)$. We define an automorphism $\sigma$ of $G$ by $\sigma(g)=\overline{g}$ and let $B:=G^{\sigma}=\Ort(n)$. Then 
$G=KBH$ is equivalent to $\min(p,q)\leq 2$ or $\min(n_1,n_2,n_3)\leq 1$. 
\end{thm}

Note that $BH/H\simeq \Gr_p(\R^n)$ is the real Gra\ss{}mannian of $p$ dimensional subspaces in $\R^n$ and that $BK/K\simeq \B_{n_1,n_1+n_2}(\R^n)$ is the real flag manifold of pairs of subspaces $(F_1,F_2)$ of dimension $n_1$ and $n_1+n_2$ respectively such that $F_1\subseteq F_2$. Using a finite dimensional approximation argument we can prove a version of Theorem \ref{thm:6.24}
 in the case of groups of operators which are Hilbert-Schmidt plus identity.

\begin{thm}\label{triun}
Let $I_1\cup I_2\cup I_3=J_1\cup J_2=I$ be partitions of a countable infinite index set $I$, let $(e_i)_{i\in I}$ be an orthonormal basis of a Hilbert space $\H$ and denote $\H_J:=\spanh\{e_j\}_{j\in J}$ for $J\subseteq I$.  Consider the naturally embedded groups $K:=\U_2(\H_{I_1})\times\U_2(\H_{I_2})\times\U_2(\H_{I_3})$ and $H:=\U_2(\H_{J_1})\times\U_2(\H_{J_2})$ in $G:=\U_2(\H)$. We define an automorphism $\sigma$ of $G$ by $\sigma(u)=\overline{u}=JuJ$, where $J$ is the canonical complex conjugation on $\H=\ell^2(I)$,  and let $B:=G^{\sigma}\cong\Ort_2(\H)$. Then $KBH$ is dense in $G$ if and only if $\min(|J_1|,|J_2|)\leq 2$ or $\min(|I_1|,|I_2|,|I_3|)\leq 1$, where $|I|$ denotes the cardinality of a set $I$. 
\end{thm}

\begin{ex}{\rm 
For a Hilbert space $\H$, we write $\P(\H)$ for its projective space, i.e., the set of all $1$-dimensional subspaces of~$\H$. The standard action of $\exp(i\ell^2(\N,\R))$ on $\P(\ell^2(\N,\C))$ is weakly visible if we take $S=\P(\ell^2(\N,\R))$ and the canonical complex conjugation on $\ell^2(\N,\C)$. This follows from the density of $KBH$ in $G$, whith $G=\U_2(\ell^2(\N,\C))$, $H=\U(1)\times\U_2(\ell^2(\N_{>1},\C))$, $K=\exp(i\ell^2(\N,\R))
\subseteq \U_2(\ell^2(\N,\C))$ and $B=G^{\sigma}$. Here $\sigma(u)=\overline{u}=JuJ$, where $J$ is the canonical complex conjugation on $\ell^2(\N,\C)$.
}\end{ex}

Lemma \ref{approxtriun} can be used to prove multiplicity-free branching rules for representations realized on spaces of holomorphic sections of line bundles over flag manifolds as in Section~VII of \cite{neeb04} or \cite{neeb12}.   Lemma~\ref{approxtriun}(3) 
can be used to prove the multiplicity-free property of tensor product representations by taking the diagonal action of a group $G$ on a product of line bundles over spaces $G/H$ and $G/K$ as in Example 2.4 of \cite{kob04}.



\section*{Acknowledgments}

K.-H.~Neeb acknowledges support by DFG-grant NE 413/9-1. 
M. Miglioli acknowledges support by a  CONICET postdoctoral fellowship and a DAAD grant for short term visit, and is grateful for 
the excellent working conditions provided by the FAU Erlangen-N\"urnberg.





\noindent
\end{document}